\documentclass[a4paprt,leqno]{article}
\usepackage[dvips]{graphicx}
\usepackage{amssymb}
\usepackage{amsmath}
\usepackage{mathrsfs}
\usepackage{color}
\usepackage{ulem}
\usepackage{comment}
\usepackage{theorem}
\usepackage{ascmac}
\newcommand{\pref}[1]{(\ref{#1})}
\newcommand{\be}{\begin{equation}}
\newcommand{\ee}{\end{equation}}

\newcommand{\qed}{{\unskip\nobreak\hfil\penalty50\quad\null\nobreak\hfil
	$\square$\parfillskip0pt\finalhyphendemerits0\par\medskip}}

\newcommand{\zume}{\!\!\!}
\newcommand{\veca}{\mbox{\boldmath $ a $}}
\newcommand{\vecb}{\mbox{\boldmath $ b $}}
\newcommand{\vecc}{\mbox{\boldmath $ c $}}
\newcommand{\vecd}{\mbox{\boldmath $ d $}}

\newcommand{\vecf}{\mbox{\boldmath $ f $}}

\newcommand{\vecu}{\mbox{\boldmath $ u $}}

\newcommand{\vectau}{\mbox{\boldmath $ \tau $}}

\renewcommand{\epsilon}{\varepsilon}
\newtheorem{defi}{Definition}[section]
\newtheorem{thm}{Theorem}[section]
\newtheorem{prop}{Proposition}[section]
\newtheorem{lem}{Lemma}[section]
\newtheorem{cor}{Corollary}[section]

\newtheorem{fact}{Fact}[section]

\theorembodyfont{\rmfamily}

\newtheorem{proof}{\normalfont\itshape Proof.}

\title{A M\"{o}bius invariant discretization
and decomposition of the M\"{o}bius energy}
\author{Simon Blatt\thanks{Paris Lodron Universit\"{a}t Salzburg,
Austria,
e-mail: Simon.Blatt@sbg.ac.at},
Aya Ishizeki\thanks{Chiba University,
Japan,
e-mail:
a.ishizeki@chiba-u.jp,
supported by KAKENHI (17J01429)}
\ \&
Takeyuki Nagasawa\thanks{Saitama University,
Japan,
e-mai:
tnagasaw@rimath.saitama-u.ac.jp,
supported by KAKENHI (17K05310).}}
\date{}
\begin{document}
\maketitle
\begin{abstract}
The M\"{o}bius energy,
defined by O'Hara,
is one of the knot energies,
and named after the M\"{o}bius invariant property which was shown by 
Freedman-He-Wang.
The energy can be decomposed into three parts,
each of which is M\"{o}bius invariant,
proved by Ishizeki-Nagasawa.
Several discrete versions of M\"{o}bius energy,
that is,
corresponding energies for polygons,
are known,
and it showed that they converge to the continuum version as the number of vertices to infinity.
However already-known discrete energies lost the property of M\"{o}bius invariance,
nor the M\"{o}bius invariant decomposition.
Here a new discretization of the M\"{o}bius energy is proposed.
It has the M\"{o}bius invariant property,
and can be decomposed into the M\"{o}bius invariant components which converge to the original components of decomposition in the continuum limit.
Though the decomposed energies are M\"{o}bius invariant,
their densities are not.
As a by-product,
it is shown that the decomposed energies have alternative representation with the M\"{o}bius invariant densities.
\end{abstract}
\section{Introduction}
Let $ \vecf $ be a closed curve in $ \mathbb{R}^n $,
parametrized by the arc-length,
with the total length $ \mathcal L $.
The M\"{o}bius energy
\[
	\mathcal{E} ( \vecf ) =
	\iint_{ ( \mathbb{R} / \mathcal{L} \mathbb{Z} )^2 }
	\left(
	\frac 1 { \| \vecf ( s_1 ) - \vecf ( s_2 ) \|_{ \mathbb{R}^n }^2 }
	-
	\frac 1 { \mathscr{D} ( \vecf ( s_1 ) , \vecf ( s_2 ) )^2 }
	\right)
	d s_1 d s_2
\]
was proposed by O'Hara \cite{OH} to determine the canonical configuration of knots,
that is,
closed curves embedded in $ \mathbb{R}^3 $.
The energy is also defined for curves in $ \mathbb{R}^n $.
In the above formula,
$ \| \vecf ( s_1 ) - \vecf ( s_2 ) \|_{ \mathbb{R}^n } $ and $ \mathscr{D} ( \vecf ( s_1 ) , \vecf ( s_2 ) $ are the extrinsic and intrinsic distance between two points $ \vecf ( s_1 ) $ and $ \vecf ( s_2 ) $ on the curve respectively.
The energy is invariant under the M\"{o}bius transformation,
and this fact played an important role to show the existence of minimizers in each prime knot class \cite{FHW}.
We consider a discrete energy,
that is,
the energy defined for polygons.
There are at least two discrete versions of the energy,
one was defined by Kim-Kusner \cite{KimKusner},
and another was by Simon \cite{Simon}.
These energies converge to the M\"{o}bius energy under the continuum limit (\cite{Rawdon-Simon,Scholtes}),
however they lost the M\"{o}bius invariance (for the definition of M\"{o}bius invariant of discrete energy,
see below,
and it is not hard work to see that Kim-Kusner's energy and Simon's energy lost the M\"{o}bius invariant properties).
Recently the authors proposed a new discrete energy and showed its convergence \cite{BIN}.
This has the M\"{o}bius invariant property.
\par
In this paper,
another discrete M\"{o}bius energy is proposed,
which keeps the M\"{o}bius invariant property and also has the M\"{o}bius invariant decomposition.
The M\"{o}bius energy can be decomposed into
\begin{align*}
	\mathcal{E} ( \vecf )
	= & \
	\mathcal{E}_1 ( \vecf )
	+
	\mathcal{E}_2 ( \vecf )
	+
	4
	,
	\\
	\mathcal{E}_1 ( \vecf )
	= & \
	\frac 12 \iint_{ ( \mathbb{R} / \mathcal{L} \mathbb{Z} )^2 }
	\frac
	{ \| \Delta \vectau \|_{ \mathbb{R}^n }^2 }
	{ \| \Delta \vecf \|_{ \mathbb{R}^n }^2 }
	\,
	d s_1 d s_2
	,
	\\
	\mathcal{E}_2 ( \vecf )
	= & \
	\iint_{ ( \mathbb{R} / \mathcal{L} \mathbb{Z} )^2 }
	\frac 2
	{ \| \Delta \vecf \|_{ \mathbb{R}^n }^2 }
	\left\langle
	\vectau ( s_1 ) \wedge \frac { \Delta \vecf } { \| \Delta \vecf \|_{ \mathbb{R}^n } }
	,
	\vectau ( s_2 ) \wedge \frac { \Delta \vecf } { \| \Delta \vecf \|_{ \mathbb{R}^n } }
	\right\rangle_{ {\bigwedge}^2 \mathbb{R}^n }
	d s_1 d s_2 ,
\end{align*}
where $ \vectau $ is the unit tangent vector,
and the notation $ \Delta $ means
\[
	\Delta v = v ( s_1 ) - v ( s_2 )
\]
for a function on $ \mathbb{R} / \mathcal{L} \mathbb{Z} $.
This was shown by the second and third authors \cite{IshizekiNagasawaI,IshizekiNagasawaIII}.
Each decomposed part is also M\"{o}bius invariant.
The newly proposed discrete energy in this paper can be decomposed into M\"{o}bius invariant components,
and each component converges to the corresponding part of the original decomposition.
\par
Let $ \mathcal{F} $ be a discrete energy for polygons in $ \mathbb{R}^n $,
and let $ T \, : \, \mathbb{R}^n \to \overline { \mathbb{R}^n } $ be a M\"{o}bius transformation.
Since the image $ T(P) $ of a polygon $ P $ is not a polygon,
we cannot define $ \mathcal{F} (T(P)) $.
Hence we define the M\"{o}bius invariance of discrete energies as follows.
\begin{defi}
Let $ \vecf \, : \, \mathbb{R}/\mathbb{Z} \to \mathbb{R}^n $ be a closed curve.
The parameter may not be necessarily arc-length one. 
For $ 0 \leqq \theta_1 < \theta_2 < \cdots < \theta_m < 1 $,
$ P_m ( \vecf , \{ \theta_i \}_{ i=1 }^m ) $ is an $ m $-polygon with the $ i $-th vertex $ \vecf ( \theta_i ) $,
simply denoted by $ P_m ( \vecf ) $ hereafter.
We say that a discrete energy $ \mathcal{F} $ is M\"{o}bius invariant if
\[
	\mathcal{F} ( P_m ( \vecf ) ) = \mathcal{F} ( P_m (T \vecf) )
\]
holds for and all $ \vecf $,
all $ \{ \theta_i \} $,
and all M\"{o}bius transformations $ T $ which maps $ \mathrm{Im} ( \vecf ) $ into $ \mathbb{R}^n $.
\label{def_Moebius_inv}
\end{defi}
\par
The main result of this paper is as follows.
\begin{thm}
There exist energies $ \mathcal{E}^m $,
$ \mathcal{E}_1^m $,
and
$ \mathcal{E}_2^m $ for $ m $-polygons such that
\begin{itemize}
\item
	$ \mathcal{E}^m = \mathcal{E}_1^m + \mathcal{E}_2^m + 4 $,
\item
	they are M\"{o}bius invariant in the sense of Definition \ref{def_Moebius_inv}.
\item
	Assume that $ f \in W^{2,\infty} $ satisfies the bi-Lipschtiz estimate.
	Let $ P_m ( \vecf ) $ be an $ m $-polygon with the $ i $-th vertex $ \vecf ( \theta_i ) $ satisfying the equi-lateral condition.
	Then
	$ \mathcal{E}^m ( P_m ( \vecf ) ) $,
	$ \mathcal{E}_1^m ( P_m ( \vecf ) ) $,
	and
	$ \mathcal{E}_2^m ( P_m ( \vecf ) ) $ converges to
	$ \mathcal{E} ( \vecf ) $,
	$ \mathcal{E}_1 ( \vecf ) $,
	and
	$ \mathcal{E}_2 ( \vecf ) $ as $ m \to \infty $ espectively.
\end{itemize}
\end{thm}
\par
Before defining discrete energies we study the M\"{o}bius invariant properties of the conformal angle in next section.
This is necessary for defining M\"{o}bius invariant discrete energies.
As a by-product,
we present an alternative proof of the M\"{o}bius invariant property of $ \mathcal{E}_1 $ and
$ \mathcal{E}_2 $.
We explain how to define the discrete energies and show their M\"{o}bius invariance in section 3.
In the final section we shall prove their convergence as $ m \to \infty $.
\section{The M\"{o}bius invariant property}
\par
\begin{defi}
Let $ C_{ij} $ be the circle through $ s_j $ which is tangent to $ \mathrm{Im} \ \vecf $ at $ \vecf( s_i ) $.
The angle between $ C_{ij} $ and $ C_{ji} $ at $ \vecf( s_i ) $ or $ \vecf ( s_j ) $ is called the {\rm conformal angle}, and we denote it $ \varphi = \varphi_f ( s_i , s_j ) $.
\end{defi}
\par
By elementary calculation we find the relation
\begin{align*}
	\cos \varphi ( s_1 , s_2 )
	= & \
	\frac { 2 ( \Delta \vecf \cdot \vectau ( s_1 ) ) ( \Delta \vecf \cdot \vectau ( s_2 ) ) }
	{ \| \Delta \vecf \|_{ \mathbb{R}^n }^2 }
	-
	\vectau ( s_1 ) \cdot \vectau ( s_2 )
	\\
	= & \
	- \| \vecf ( s_1 ) - \vecf ( s_2 ) \|_{ \mathbb{R}^n }^2
	\left( \mathscr{M}_1 ( \vecf ) + \mathscr{M}_2 ( \vecf ) \right)
	+
	1 ,
\end{align*}
where $ \mathscr{M}_i ( \vecf ) $ is the energy density of $ \mathcal{E}_i ( \vecf ) $.
From this,
we easily obtain an expression of the M\"{o}bius energy by using the conformal angle,
which is known as the cosine formula.
\begin{prop}[the cosine formula]
It holds that
\[
	\iint_{ ( \mathbb{R} / \mathcal{L} \mathbb{Z} )^2 }
	\frac { 1 - \cos \varphi ( s_1 , s_2 ) }
	{
	\| \vecf ( \theta_1 ) - \vecf ( \theta_2 ) \|_{ \mathbb{R}^n }^2
	}
	\, d s_1 d s_2
	=
	\mathcal{E} ( \vecf ) - 4
	.
\]
\end{prop}
\par
The M\"{o}bius invariance of $ \varphi $ follows from that of cross ratio.
\begin{lem}
It holds that
\[
	\cos \varphi ( \theta_1 , \theta_2 )
	=
	\frac
	{ \| \Delta \vecf \|_{ \mathbb{R}^n }^2 }
	{
	2
	\| \dot { \vecf } ( \theta_1 ) \|_{ \mathbb{R}^n }
	\| \dot { \vecf } ( \theta_2 ) \|_{ \mathbb{R}^n }
	}
	\frac \partial { \partial \theta_1 \partial \theta_2 }
	\frac
	{ \| \Delta \vecf \|_{ \mathbb{R}^n }^2 }
	{
	\| \dot { \vecf } ( \theta_1 ) \|_{ \mathbb{R}^n }
	\| \dot { \vecf } ( \theta_2 ) \|_{ \mathbb{R}^n }
	}
	.
\]
In particular,
$ \cos \varphi ( \theta_1 , \theta_2 ) $ is invariant under M\"{o}bius transformations.
\end{lem}
\begin{proof}
The direct calculation derives
\begin{align*}
	\frac \partial { \partial \theta_1 \partial \theta_2 }
	\log \| \Delta \vecf \|_{ \mathbb{R}^n }^2
	= & \
	\frac {
	2
	\| \dot { \vecf } ( \theta_1 ) \|_{ \mathbb{R}^n }
	\| \dot { \vecf } ( \theta_2 ) \|_{ \mathbb{R}^n }
	}
	{ \| \Delta \vecf \|_{ \mathbb{R}^n }^2 }
	\cos \varphi ( \theta_1 , \theta_2 )
	.
\end{align*}
Since
\[
	\frac \partial { \partial \theta_1 \partial \theta_2 }
	\log
	\| \dot { \vecf } ( \theta_1 ) \|_{ \mathbb{R}^n }
	\| \dot { \vecf } ( \theta_2 ) \|_{ \mathbb{R}^n }
	=
	\frac \partial { \partial \theta_1 \partial \theta_2 }
	\left(
	\log
	\| \dot { \vecf } ( \theta_1 ) \|_{ \mathbb{R}^n }
	+
	\log
	\| \dot { \vecf } ( \theta_2 ) \|_{ \mathbb{R}^n }
	\right)
	= 0
	,
\]
we have
\[
	\frac \partial { \partial \theta_1 \partial \theta_2 }
	\log \| \Delta \vecf \|_{ \mathbb{R}^n }^2
	=
	\frac \partial { \partial \theta_1 \partial \theta_2 }
	\log
	\frac
	{ \| \Delta \vecf \|_{ \mathbb{R}^n }^2 }
	{ \| \dot { \vecf } ( \theta_1 ) \|_{ \mathbb{R}^n } \| \dot { \vecf } ( \theta_2 ) \|_{ \mathbb{R}^n } }
	.
\]
Therefore we get the first half of assertion.
Since a M\"{o}bius transformation keeps the cross ration of distinct four points in $ \mathbb{R}^d $,
it holds that
\[
	\frac{
	\| \vecf ( \theta_1 + \delta \theta_1 ) - \vecf ( \theta_1 ) \|_{ \mathbb{R}^n }
	\| \vecf ( \theta_2 + \delta \theta_2 ) - \vecf ( \theta_2 ) \|_{ \mathbb{R}^n }
	}
	{
	\| \vecf ( \theta_1 + \delta \theta_1 ) - \vecf ( \theta_2 + \delta \theta_2 ) \|_{ \mathbb{R}^n }
	\| \vecf ( \theta_1 ) - \vecf ( \theta_2 ) \|_{ \mathbb{R}^n }
	}
\]
is a M\"{o}bius invariance for sufficiently small $ | \delta \theta_1 | $ and $ | \delta \theta_2 | $ provided $ \vecf ( \theta_1 ) \ne \vecf ( \theta_2 ) $.
Dividing by $ \delta \theta_1 \delta \theta_2 $ and passing $ \delta \theta_1 \to 0 $,
$ \delta \theta_2 \to 0 $,
we find the quantity $ \displaystyle{
	\frac{
	\| \dot { \vecf } ( \theta_1 ) \|_{ \mathbb{R}^n }
	\| \dot { \vecf } ( \theta_2 ) \|_{ \mathbb{R}^n }
	}
	{ \| \Delta \vecf \|_{ \mathbb{R}^n }^2 }
} $ is also a M\"{o}bius invariance.
Consequently,
the second assertion has been shown.
\qed
\end{proof}
\par
Put
\[
	g = g ( \theta_1 , \theta_2 )
	=
	\frac{
	\| \dot { \vecf } ( \theta_1 ) \|_{ \mathbb{R}^n }
	\| \dot { \vecf } ( \theta_2 ) \|_{ \mathbb{R}^n }
	}
	{ \| \Delta \vecf \|_{ \mathbb{R}^n }^2 } .
\]
Note that we can define $ \displaystyle{ \frac { \partial^2 g } { \partial \theta_1 \partial \theta_2 } } $ in the classical sense for $ \vecf \in C^2 ( \mathbb{R} / \mathbb{Z} ) $,
in the weak sense for $ \vecf \in W^{2,1}_{\scriptsize\rm loc} ( \mathbb{R} / \mathbb{Z} \setminus \{ \theta_1 = \theta_2 \} ) $. 
The density of the M\"{o}bius energy can be written by the M\"{o}bius invariance $ g $ and its derivatives:
\[
	\mathcal{E} ( \vecf ) - 4
	=
	\iint_{ ( \mathbb{R}/ \mathbb{Z} )^2 }
	\left\{
	\frac 1g
	-
	\frac 12
	\frac \partial { \partial \theta_1 }
	\left(
	\frac 1g
	\frac { \partial g } { \partial \theta_2 }
	\right)
	\right\}
	d \theta_1 d \theta_2
	.
\]
This gives an alternative proof of the M\"{o}bius invariance of decomposed energies.
It is not difficult to see
\[
	\frac 1g
	-
	\frac 12
	\frac \partial { \partial \theta_1 }
	\left(
	\frac 1g
	\frac { \partial g } { \partial \theta_2 }
	\right)
	=
	\frac 1g
	\left(
	1
	+
	\frac 12
	\frac { \partial^2 g } { \partial \theta_1 \partial \theta_2 }
	\right)
	-
	\frac 1 { 2 g^2 }
	\det
	\left(
	\begin{array}{cc}
	\partial_1 \partial_2 g & \partial_1 g \\
	\partial_2 g & 2 g
	\end{array}
	\right)
	.
\]
We define the energies $ \widehat { \mathcal{E}_1 } $ and $ \widehat { \mathcal{E}_2 } $ by
\begin{align*}
	\widehat { \mathcal{E} }_1 ( \vecf )
	= & \
	\iint_{ ( \mathbb{R}/ \mathbb{Z} )^2 }
	\frac 1g
	\left(
	1
	+
	\frac 12
	\frac { \partial^2 g } { \partial \theta_1 \partial \theta_2 }
	\right)
	d \theta_1 d \theta_2
	,
	\\
	\widehat { \mathcal{E} }_2 ( \vecf )
	= & \
	-
	\iint_{ ( \mathbb{R}/ \mathbb{Z} )^2 }
	\frac 1 { 2 g^2 }
	\det
	\left(
	\begin{array}{cc}
	\partial_1 \partial_2 g & \partial_1 g \\
	\partial_2 g & 2 g
	\end{array}
	\right)
	d \theta_1 d \theta_2
	.
\end{align*}
These are M\"{o}bius invariant energies,
since so is $ g $.
\begin{thm}
We have
$ \mathcal{E}_1 ( \vecf ) = \widehat { \mathcal{E} }_1 ( \vecf ) $,
and
$ \mathcal{E}_2 ( \vecf ) = \widehat { \mathcal{E} }_2 ( \vecf ) $.
\label{E=hat E}
\end{thm}
\begin{proof}
Since $ \mathcal{E}_1 + \mathcal{E}_2 = \widehat { \mathcal{E} }_1 + \widehat { \mathcal{E}_2 } = \mathcal{E} - 4 $,
it is enough to show the relation for $ \mathcal{E}_1 ( \vecf ) $.
Put
\[
	g_n = \| \Delta \vecf \|_{ \mathbb{R}^n }^2
	,
	\quad
	g_d =
	\| \dot { \vecf } ( \theta_1 ) \|_{ \mathbb{R}^n }
	\| \dot { \vecf } ( \theta_2 ) \|_{ \mathbb{R}^n }
	.
\]
Then we have
$ \displaystyle{
	g = \frac { g_n } { g_d } ,
} $
and
\begin{align*}
	\mathcal{E}_1 ( \vecf )
	= & \
	\frac 12
	\iint_{ ( \mathbb{R} / \mathcal{L} \mathbb{Z} )^2 }
	\frac { \| \vectau ( s_1 ) - \vectau ( s_2 ) \|_{ \mathbb{R}^n }^2 }
	{ \| \Delta \vecf \|_{ \mathbb{R}^n }^2 }
	\, d s_1 d s_2
	\\
	= & \
	\iint_{ ( \mathbb{R} / \mathcal{L} \mathbb{Z} )^2 }
	\frac { 1 - \vectau ( s_1 ) \cdot \vectau ( s_2 ) }
	{ \| \Delta \vecf \|_{ \mathbb{R}^n }^2 }
	\, d s_1 d s_2
	\\
	= & \
	\iint_{ ( \mathbb{R} / \mathbb{Z} )^2 }
	\frac 1g
	\left( 1 - 
	\frac { \dot { \vecf } ( \theta_1 ) \cdot \dot { \vecf } ( \theta_2 ) }
	{
	\| \dot { \vecf } ( \theta_1 ) \|_{ \mathbb{R}^n }
	\| \dot { \vecf } ( \theta_2 ) \|_{ \mathbb{R}^n }
	}
	\right)
	d \theta_1 d \theta_2
	\\
	= & \
	\iint_{ ( \mathbb{R} / \mathbb{Z} )^2 }
	\frac 1g
	\left(
	1 +
	\frac 1
	{
	2
	\| \dot { \vecf } ( \theta_1 ) \|_{ \mathbb{R}^n }
	\| \dot { \vecf } ( \theta_2 ) \|_{ \mathbb{R}^n }
	}
	\frac { \partial^2 } { \partial \theta_1 \partial \theta_2 }
	\| \Delta \vecf \|_{ \mathbb{R}^n }^2
	\right)
	d \theta_1 d \theta_2
	\\
	= & \
	\iint_{ ( \mathbb{R} / \mathbb{Z} )^2 }
	\frac { g_d } { g_n }
	\left(
	1 +
	\frac 1 { 2 g_d }
	\frac { \partial^2 g_n } { \partial \theta_1 \partial \theta_2 }
	\right)
	d \theta_1 d \theta_2
	.
\end{align*}
On the other hand,
we have
\[
	\widehat { \mathcal{E} }_1 ( \vecf )
	=
	\iint_{ ( \mathbb{R} / \mathbb{Z} )^2 }
	\frac { g_d } { g_n }
	\left(
	1 +
	\frac 12
	\frac { \partial^2 } { \partial \theta_1 \partial \theta_2 }
	\frac { g_n } { g_d }
	\right)
	d \theta_1 d \theta_2 .
\]
Therefore the difference between $ \mathcal{E}_1 ( \vecf ) $ and $ \widehat { \mathcal{E} }_1 ( \vecf ) $ is
\begin{align*}
	&
	\mathcal{E}_1 ( \vecf )
	-
	\widehat { \mathcal{E} }_1 ( \vecf )
	\\
	& \quad
	=
	\frac 12
	\iint_{ ( \mathbb{R} / \mathbb{Z} )^2 }
	\frac { g_d } { g_n }
	\left(
	\frac 1 { g_d }
	\frac { \partial^2 g_n } { \partial \theta_1 \partial \theta_2 }
	-
	\frac { \partial^2 } { \partial \theta_1 \partial \theta_2 }
	\frac { g_n } { g_d }
	\right)
	d \theta_1 d \theta_2
	\\
	& \quad
	=
	\frac 12
	\iint_{ ( \mathbb{R} / \mathbb{Z} )^2 }
	\left\{
	\left(
	\frac \partial { \partial \theta_1 } \log g_d
	\right)
	\left(
	\frac \partial { \partial \theta_2 } \log g_n
	\right)
	+
	\left(
	\frac \partial { \partial \theta_1 } \log g_n
	\right)
	\left(
	\frac \partial { \partial \theta_2 } \log g_d
	\right)
	\right.
	\\
	& \quad \qquad \qquad \qquad \qquad
	\left.
	- \,
	2
	\left(
	\frac \partial { \partial \theta_1 } \log g_d
	\right)
	\left(
	\frac \partial { \partial \theta_2 } \log g_d
	\right)
	+
	\frac 1 { g_d }
	\frac { \partial^2 g_d } { \partial \theta_1 \partial \theta_2 }
	\right\}
	d \theta_1 d \theta_2
	.
\end{align*}
We can show
\[
	\iint_{ ( \mathbb{R} / \mathbb{Z} )^2 }
	\left\{
	\left(
	\frac \partial { \partial \theta_1 } \log g_d
	\right)
	\left(
	\frac \partial { \partial \theta_2 } \log g_n
	\right)
	+
	\left(
	\frac \partial { \partial \theta_1 } \log g_n
	\right)
	\left(
	\frac \partial { \partial \theta_2 } \log g_d
	\right)
	\right\}
	d \theta_1 d \theta_2
	=
	0
\]
in a manner similar to the proof of \cite[Theorem 3.2]{IshizekiNagasawaI}.
Since $ \log g_d $ does not have singularity anywhere,
\begin{align*}
	\iint_{ ( \mathbb{R} / \mathbb{Z} )^2 }
	\left(
	\frac \partial { \partial \theta_1 } \log g_d
	\right)
	\left(
	\frac \partial { \partial \theta_2 } \log g_d
	\right)
	d \theta_1 d \theta_2
	=
	0 .
\end{align*}
We can perform the integration by parts to get
\[
	\iint_{ ( \mathbb{R} / \mathbb{Z} )^2 }
	\frac 1 { g_d }
	\frac { \partial^2 g_d } { \partial \theta_1 \partial \theta_2 }
	d \theta_1 d \theta_2
	=
	\iint_{ ( \mathbb{R} / \mathbb{Z} )^2 }
	\left(
	\frac \partial { \partial \theta_1 } \log g_d
	\right)
	\left(
	\frac \partial { \partial \theta_2 } \log g_d
	\right)
	d \theta_1 d \theta_2
	=
	0 .
\]
\qed
\end{proof}
\par
Though Theorem \ref{E=hat E} is not our main purpose of this paper,
it seems to be interesting.
The decomposed energies $ \mathcal{E}_i $ are M\"{o}bius invariant,
but their densities are not.
Theorem \ref{E=hat E} says that $ \mathcal{E}_i $ can be rewritten as the energies with M\"{o}bius invariant densities.
The domain of $ \widehat { \mathcal{E} }_i $ is narrower that that of $ \mathcal{E}_i $,
hence we can regard $ \mathcal{E}_i $ as a relaxation of $ \widehat{ \mathcal{E} }_i $.
\section{A M\"{o}bius invariant discritization}
\par
For the discretization we use notation $ \Delta_i^j $ and $ \Delta_i $ to mean
\[
	\Delta_i^j \vecf = \vecf ( \theta_j ) - \vecf ( \theta_i ) ,
	\quad
	\Delta_i \vecf = \Delta_i^{ i+1 } \vecf = \vecf ( \theta_{ i+1 } ) - \vecf ( \theta_i ) .
\]
We discritize the M\"{o}bius invariance $ g $ as
\[
	g_{ij}
	=
	\frac
	{
	\| \Delta_i^j \vecf \|_{ \mathbb{R}^n }
	\| \Delta_{ i+1 }^{ j+1 } \vecf \|_{ \mathbb{R}^n }
	}
	{
	\| \Delta_i \vecf \|_{ \mathbb{R}^n }
	\| \Delta_j \vecf \|_{ \mathbb{R}^n } .
	}
\]
Since it is the cross ratio of for points $ \vecf ( \theta_i ) $,
$ \vecf ( \theta_{ i+1 } ) $,
$ \vecf ( \theta_j ) $,
and $ \vecf ( \theta_{ j+1 } ) $,
it is also a M\"{o}bius invariance.
\par
Using $ g_{ij} $ and difference operation instead of $ g $ and differential operation,
we consider discrete energies
\begin{align*}
	\mathcal{E}_1^m ( \vecf )
	= & \
	\sum_{ i \ne j }
	\frac 1 { g_{ij} }
	\left( 1 +\frac 12 \Delta_i \Delta_j g_{ij} \right)
	,
	\\
	\mathcal{E}_2^m ( \vecf )
	= & \
	-
	\sum_{ i \ne j }
	\frac 1 { 2 g_{ij} }
	\det
	\left(
	\begin{array}{cc}
	\Delta_i \Delta_j g_{ij} & \Delta_i g_{ij} \\
	\Delta_j g_{ij} & 2 g_{ij}
	\end{array}
	\right)
	,
	\\
	\mathcal{E}^m ( \vecf )
	= & \
	\mathcal{E}_1^m ( \vecf )
	+
	\mathcal{E}_2^m ( \vecf )
	+
	4
	.
\end{align*}
Since $ \mathcal{E}_1^m $ and $ \mathcal{E}_2^m $ are a discrete version of $ \mathcal{E}_1 $ and $ \mathcal{E}_2 $,
they are expected to convergent as $ m \to \infty $ under suitable assumption.
Indeed $ \mathcal{E}_1^m $ converges to $ \mathcal{E}_1 $,
but $ \mathcal{E}_2^m $,
and therefore $ \mathcal{E}^m $ do not.
A regular $ m $-polygon converges to a right circle,
but
\[
	\mathcal{E}^m ( \mbox{a regular $m$-polygon} ) \to \infty
	\ne
	4 = \mathcal{E} ( \mbox{a right circle} )
\]
as $ m \to \infty $.
To recover this situation,
we modify the definition as follows.
\begin{defi}
We define
$ \mathcal{E}_1^m $,
$ \mathcal{E}_2^m $,
and $ \mathcal{E}^m $ by
\begin{align*}
	\mathcal{E}_1^m ( \vecf )
	= & \
	\sum_{ i \ne j }
	\frac 1 { g_{ij} }
	\left( 1 +\frac 12 \Delta_i \Delta_j g_{ij} \right)
	,
	\\
	\mathcal{E}_2^m ( \vecf )
	= & \
	-
	\sum_{ i \ne j }
	\frac 1 { 2 g_{ij} }
	\left\{
	\det
	\left(
	\begin{array}{cc}
	\Delta_i \Delta_j g_{ij} & \Delta_i g_{ij} \\
	\Delta_j g_{ij} & 2 g_{ij}
	\end{array}
	\right)
	+
	1
	\right\}
	,
	\\
	\mathcal{E}^m ( \vecf )
	= & \
	\mathcal{E}_1^m ( \vecf )
	+
	\mathcal{E}_2^m ( \vecf )
	+
	4
	.
\end{align*}
\end{defi}
\par
\begin{prop}
The energies $ \mathcal{E}_1^m ( \vecf ) $,
$ \mathcal{E}_2^m ( \vecf ) $ and $ \mathcal{E}^m ( \vecf ) $ are M\"{o}bius invariant.
\end{prop}
\begin{proof}
Their M\"{o}bius invariance follows from that of $ g_{ij} $.
\qed
\end{proof}
\section{Convergence of discrete energies under the equi-lateral condition}
\par
In this section we would like to show the convergence of $ \mathcal{E}_1^m $ and $ \mathcal{E}^m - 4 $.
As a corollary we get the convergence of $ \mathcal{E}_2^m $ also.
\subsection{Preparations}
\par
We put for a positive sequence $ u = \{ u_i \} $
\[
	\begin{array}{rll}
	A_i ( u ) = & \zume
	\displaystyle{ \frac { u_i + u_{ i+1 } } 2 }
	\quad
	& \mbox{(the arithmetic mean)} ,
	\\
	G_i ( u ) = & \zume
	\displaystyle{ \sqrt{ u_i u_{ i+1 } } }
	& \mbox{(the geomeric mean)},
	\\
	H_i (u) = & \zume
	\displaystyle{ 2 \left( \frac 1 { u_i } + \frac 1 { u_{ i+1 } } \right)^{-1} }
	& \mbox{(the harmonic mean)} .
	\end{array}
\]
We use similar notations for a positive sequence with two subscripts:
\[
	\begin{array}{rlrlrl}
	A_i (u)
	= & \zume
	\displaystyle{
	\frac { u_{ij} + u_{(i+1),j} } 2
	}
	,
	\
	&
	G_i (u)
	= & \zume
	\displaystyle{
	\sqrt{ u_{ij} u_{(i+1)j} }
	}
	,
	\
	&
	H_i (u)
	= & \zume
	\displaystyle{
	2 \left( \frac 1 { u_{ij} } + \frac 1 { u_{(i+1)j} } \right)^{-1}
	}
	,
	\\
	A_j (u)
	= & \zume
	\displaystyle{
	\frac { u_{ij} + u_{i(j+1)} } 2
	}
	,
	\
	&
	G_j (u)
	= & \zume
	\displaystyle{
	\sqrt{ u_{ij} u_{i(j+1)} }
	}
	,
	\
	&
	H_j (u)
	= & \zume
	\displaystyle{
	2 \left( \frac 1 { u_{ij} } + \frac 1 { u_{i(j+1)} } \right)^{-1}
	,
	}
	\end{array}
\]
and
\begin{align*}
	A_{ij} (u)
	= & \
	\frac { u_{ij} + u_{ (i+1)j } + u_{ i(j+1) } + u_{ (i+1)(j+1) } } 4
	,
	\\
	G_{ij} (u)
	= & \
	\sqrt[4]{ u_{ij} u_{ (i+1)j } u_{ i(j+1) } u_{ (i+1)(j+1) } }
	,
	\\
	H_{ij} (u)
	= & \
	4 \left(
	\frac 1 { u_{ij} } + \frac 1 { u_{ (i+1)j } }
	+
	\frac 1 { u_{ i(j+1) } } + \frac 1 { u_{ (i+1)(j+1) } }
	\right)^{-1}
	.
\end{align*}
In most cases we use,
for example,
$ A_i ( u_{ij} ) $ or $ A_i u_{ij} $ instead of $ A_i ( \{ u_{ij} \} ) $.
Furtheremore $ A_i G_j ( u_{ij} ) = A_i ( G_j ( u_{ij} ) ) $,
and so on.
It is easy to see the next identities.
\begin{fact}
There hold
\begin{align*}
	\Delta_i ( u_i v_i ) = & \
	( \Delta_i u_i ) A_i (v) + A_i (u) ( \Delta_i v_i ) ,
	\\
	\Delta_i \frac { u_i } { v_i } = & \
	\frac { \Delta_i u_i } { H_i (v) } + A_i (u) \left( \Delta_i \frac 1 v_i \right)
	,
	\\
	A_i \left( A_j ( \{ u_{ij} \} ) v_{ij} \right)
	= & \
	\frac { \Delta_i \Delta_j u_{ij} }
	{ H_{ij} v }
	+
	\Delta_j ( A_i u ) \Delta_i \left( \frac 1 { H_j v } \right)
	\\
	& \quad
	+ \,
	\Delta_i ( A_j u ) \Delta_j \left( \frac 1 { H_i v } \right)
	+
	( A_{ij} u ) \Delta_i \Delta_j \frac 1 { v_{ij} }
	.
\end{align*}
\label{fact}
\end{fact}
\subsection{Decomposition of densities of discrete energies}
\par
Put
\[
	g_{ij,d}
	=
	\| \Delta_i \vecf \|_{ \mathbb{R}^n }
	\| \Delta_j \vecf \|_{ \mathbb{R}^n }
	,
	\quad
	g_{ij,n}
	=
	\| \Delta_i^j \vecf \|_{ \mathbb{R}^n }
	\| \Delta_{ i+1 }^{ j+1 } \vecf \|_{ \mathbb{R}^n } ,
	\quad
	\bar g_{ij,n}
	= \| \Delta_i^j \vecf \|_{ \mathbb{R}^n }^2 .
\]
Let $ \mathscr{M}_{1,ij}^m $ and $ \mathscr{M}_{ij}^m $ be the densities of $ \mathcal{E}_1^m $ and $ \mathcal{E}^m - 4 $ respectively:
\[
	\mathcal{E}_1^m ( \vecf ) = \sum_{ i \ne j } \mathscr{M}_{1,ij}^m ( \vecf ) ,
	\quad
	\mathcal{E}^m ( \vecf ) - 4 = \sum_{ i \ne j } \mathscr{M}_{ij}^m ( \vecf ) ,
\]
Then
\begin{align*}
	\mathscr{M}_{1,ij}^m ( \vecf )
	= & \
	\frac { g_{ij,d} } { g_{ij,n} }
	\left( 1 + \frac 12 \Delta_i \Delta_j \frac { g_{ij,n} } { g_{ij,d} }\right)
	,
	\\
	\mathscr{M}_{ij}^m ( \vecf )
	= & \
	\frac { g_{ij,d} } { g_{ij,n} }
	\left[
	1
	-
	\frac 12 \Delta_i \Delta_j \frac { g_{ij,n} } { g_{ij,d} }
	+
	\frac { g_{ij,d} } { 2 g_{ij,n} }
	\left\{
	\left( \Delta_i \frac { g_{ij,n} } { g_{ij,d} } \right)
	\left( \Delta_j \frac { g_{ij,n} } { g_{ij,d} } \right)
	- 1
	\right\}
	\right]
	.
\end{align*}
We decompose them into
\[
	\mathscr{M}_{1,ij} ( \vecf )
	=
	\mathscr{P}_{1,ij} ( \vecf )
	+
	\mathscr{R}_{1,ij} ( \vecf ) ,
	\quad
	\mathscr{M}_{ij} ( \vecf )
	=
	\mathscr{P}_{ij} ( \vecf )
	+
	\mathscr{R}_{ij} ( \vecf )
	,
\]
where
\begin{align*}
	\mathscr{P}_{1,ij} ( \vecf )
	= & \
	\frac 12
	\frac { g_{ij,d} } { g_{ij,n} }
	\left\{
	\left(
	1 -
	\frac { \Delta_i \vecf } { \| \Delta_i \vecf \|_{ \mathbb{R}^n } }
	\cdot
	\frac { \Delta_j \vecf } { \| \Delta_j \vecf \|_{ \mathbb{R}^n } }
	\right)
	+
	\left(
	1 -
	\frac { \Delta_{ i+1 } \vecf } { \| \Delta_{ i+1 } \vecf \|_{ \mathbb{R}^n } }
	\cdot
	\frac { \Delta_{ j+1 } \vecf } { \| \Delta_{ j+1 } \vecf \|_{ \mathbb{R}^n } }
	\right)
	\right\}
	,
	\\
	\mathscr{P}_{ij} ( \vecf )
	= & \
	\frac { g_{ij,d} } { \bar g_{ij,n} }
	\left[
	1
	-
	\frac 
	1
	{ 2 g_{ij,d} }
	\left\{
	\Delta_i \Delta_j \bar g_{ij,n}
	+
	\frac { 4 \left( \Delta_i^j \vecf \cdot \Delta_i \vecf \right) \left( \Delta_i^j \vecf \cdot \Delta_j \vecf \right) }
	{ \bar g_{ij,n} }
	\right\}
	\right]
	.
\end{align*}
Then we will show
\be
	\sum_{ i \ne j }
	\mathscr{P}_{1,ij} ( \vecf )
	\to
	\mathcal{E}_1 ( \vecf )
	,
	\quad
	\sum_{ i \ne j }
	\mathscr{R}_{1,ij} ( \vecf )
	\to 0
	,
	\label{convergence_of_the_discrete_1st_energy}
\ee
\be
	\sum_{ i \ne j }
	\mathscr{P}_{ij} ( \vecf )
	\to
	\mathcal{E} ( \vecf ) - 4
	,
	\quad
	\sum_{ i \ne j }
	\mathscr{R}_{ij} ( \vecf )
	\to 0
	\label{convergence_of_the_discrete_energy}
\ee
as $ m \to \infty $.
\par
First of all,
we clarify how $ \mathscr{P}_{1,ij} $ comes from $ \mathscr{M}_{1,ij} $.
Using Fact \ref{fact},
we decompose $ \mathscr{M}_{1,ij}^m ( \vecf ) $ into four parts as
\begin{align*}
	\mathscr{M}_{1,ij}^m ( \vecf )
	= & \
	\mathrm{I} + \mathrm{II} + \mathrm{III} + \mathrm{IV}
	,
	\\
	\mathrm{I}
	= & \
	\frac { g_{ij,d} } { g_{ij,n} }
	\left\{
	1 +
	\frac { \Delta_i \Delta_j g_{ij,n} }
	{ 2 ( H_i \| \Delta_i \vecf \|_{ \mathbb{R}^n } ) ( H_j \| \Delta_j \vecf \|_{ \mathbb{R}^n } ) }
	\right\}
	,
	\\
	\mathrm{II}
	= & \
	\frac { g_{ij,d} } { g_{ij,n} }
	\frac { \Delta_j ( A_i g_{ij,n} ) } { 2 H_j \| \Delta_j \vecf \|_{ \mathbb{R}^n } }
	\Delta_i \left( \frac 1 { H_i \| \Delta_i \vecf \|_{ \mathbb{R}^n } } \right)
	,
	\\
	\mathrm{III}
	= & \
	\frac { g_{ij,d} } { g_{ij,n} }
	\frac { \Delta_i ( A_j g_{ij,n} ) } { 2 H_i \| \Delta_i \vecf \|_{ \mathbb{R}^n } }
	\Delta_j \left( \frac 1 { H_j \| \Delta_j \vecf \|_{ \mathbb{R}^n } } \right)
	,
	\\
	\mathrm{IV}
	= & \
	\frac { g_{ij,d} } { g_{ij,n} }
	\frac { A_{ij} g_{ij,n} } 2
	\left\{ \Delta_i \left( \frac 1 { \| \Delta_i \vecf \|_{ \mathbb{R}^n } } \right) \right\}
	\left\{ \Delta_j \left( \frac 1 { \| \Delta_j \vecf \|_{ \mathbb{R}^n } } \right) \right\}
	.
\end{align*}
Next we decompose $ \mathrm{I} $ into two parts.
Using Fact \ref{fact},
we have
\begin{align*}
	\Delta_j g_{ij,n}
	= & \
	\left( \Delta_j \| \Delta_i^j \vecf \|_{ \mathbb{R}^n } \right)
	\left( A_j \| \Delta_{ i+1 }^{ j+1 } \vecf \|_{ \mathbb{R}^n } \right)
	+
	\left( A_j \| \Delta_i^j \vecf \|_{ \mathbb{R}^n } \right)
	\left( \Delta_j \| \Delta_{ i+1 }^{ j+1 } \vecf \|_{ \mathbb{R}^n } \right)
	,
\end{align*}
\begin{align*}
	\Delta_i \Delta_j g_{ij,n}
	= & \
	\left( \Delta_i \Delta_j \| \Delta_i^j \vecf \|_{ \mathbb{R}^n } \right)
	\left( A_{ij} \| \Delta_{ i+1 }^{ j+1 } \vecf \|_{ \mathbb{R}^n } \right)
	\\
	& \quad
	+ \,
	\left( \Delta_j A_i \| \Delta_i^j \vecf \|_{ \mathbb{R}^n } \right)
	\left( \Delta_i A_j \| \Delta_{ i+1 }^{ j+1 } \vecf \|_{ \mathbb{R}^n } \right)
	\\
	& \quad
	+ \,
	\left( \Delta_i A_j \| \Delta_i^j \vecf \|_{ \mathbb{R}^n } \right)
	\left( \Delta_j A_i \| \Delta_{ i+1 }^{ j+1 } \vecf \|_{ \mathbb{R}^n } \right)
	\\
	& \quad
	+ \,
	\left( A_{ij} \| \Delta_i^j \vecf \|_{ \mathbb{R}^n } \right)
	\left( \Delta_i \Delta_j \| \Delta_{ i+1 }^{ j+1 } \vecf \|_{ \mathbb{R}^n } \right)
	.
\end{align*}
It holds that
\begin{align*}
	\Delta_i
	\| \Delta_i^j \vecf \|_{ \mathbb{R}^n }^2
	= & \
	- 2 A_i \Delta_i^j \vecf \cdot \Delta_i \vecf
	.
\end{align*}
On the other hand,
\begin{align*}
	\Delta_i
	\| \Delta_i^j \vecf \|_{ \mathbb{R}^n }^2
	= & \
	2 A_i \| \Delta_i^j \vecf \|_{ \mathbb{R}^n }
	\Delta_i \| \Delta_i^j \vecf \|_{ \mathbb{R}^n }
	.
\end{align*}
Therefore we obtain
\[
	\Delta_i
	\| \Delta_i^j \vecf \|_{ \mathbb{R}^n }
	=
	-
	\frac { A_i \Delta_i^j \vecf \cdot \Delta_i \vecf }
	{ A_i \| \Delta_i^j \vecf \|_{ \mathbb{R}^n } }
	.
\]
Similarly we have
\[
	\Delta_i
	\| \Delta_{ i+1 }^{ j+1 } \vecf \|_{ \mathbb{R}^n }
	=
	-
	\frac { A_i \Delta_{ i+1 }^{ j+1 } \vecf \cdot \Delta_{ i+1 } \vecf }
	{ A_i \| \Delta_{ i+1 }^{ j+1 } \vecf \|_{ \mathbb{R}^n } }
	,
\]
\[
	\Delta_j
	\| \Delta_i^j \vecf \|_{ \mathbb{R}^n }
	=
	\frac { A_j \Delta_i^j \vecf \cdot \Delta_j \vecf }
	{ A_j \| \Delta_i^j \vecf \|_{ \mathbb{R}^n } }
	,
	\quad
	\Delta_j
	\| \Delta_{ i+1 }^{ j+1 } \vecf \|_{ \mathbb{R}^n }
	=
	\frac { A_j \Delta_{ i+1 }^{ j+1 } \vecf \cdot \Delta_{ j+1 } \vecf }
	{ A_j \| \Delta_{ i+1 }^{ j+1 } \vecf \|_{ \mathbb{R}^n } }
	.
\]
It follows from
\[
	\Delta_j \| \Delta_i^j \vecf \|_{ \mathbb{R}^n }^2
	=
	2 A_j \Delta_i^j \vecf \cdot \Delta_j \vecf
\]
that
\[
	\Delta_i \Delta_j \| \Delta_i^j \vecf \|_{ \mathbb{R}^n }^2
	=
	2 A_j \Delta_i \Delta_i^j \vecf \cdot \Delta_j \vecf
	=
	- 2 \Delta_i \vecf \cdot \Delta_j \vecf .
\]
On the other hand,
\[
	\Delta_j \| \Delta_i^j \vecf \|_{ \mathbb{R}^n }^2
	=
	2 A_j \| \Delta_i^j \vecf \|_{ \mathbb{R}^n }
	\Delta_j \| \Delta_i^j \vecf \|_{ \mathbb{R}^n }
\]
derives
\begin{align*}
	\Delta_i \Delta_j \| \Delta_i^j \vecf \|_{ \mathbb{R}^n }^2
	= & \
	2
	\left( \Delta_i A_j \| \Delta_i^j \vecf \|_{ \mathbb{R}^n } \right)
	\left( \Delta_j A_i \| \Delta_i^j \vecf \|_{ \mathbb{R}^n } \right)
	\\
	& \quad
	+
	2
	\left( A_{ij} \| \Delta_i^j \vecf \|_{ \mathbb{R}^n } \right)
	\left( \Delta_i \Delta_i \| \Delta_i^j \vecf \|_{ \mathbb{R}^n } \right)
	.
\end{align*}
Hence
\[
	\Delta_i \Delta_j
	\| \Delta_i^j \vecf \|_{ \mathbb{R}^n }
	=
	-
	\frac
	{
	\Delta_i \vecf \cdot \Delta_j \vecf
	+
	\left( \Delta_i A_j \| \Delta_i^j \vecf \|_{ \mathbb{R}^n } \right)
	\left( \Delta_j A_i \| \Delta_i^j \vecf \|_{ \mathbb{R}^n } \right)
	}
	{ A_{ij} \| \Delta_i^j \vecf \|_{ \mathbb{R}^n } }
	.
\]
We obtain
\[
	\Delta_i \Delta_j
	\| \Delta_{ i+1 }^{ j+1 } \vecf \|_{ \mathbb{R}^n }
	=
	-
	\frac
	{
	\Delta_{ i+1 } \vecf \cdot \Delta_{ j+1 } \vecf
	+
	\left( \Delta_i A_j \| \Delta_{ i+1 }^{ j+1 } \vecf \|_{ \mathbb{R}^n } \right)
	\left( \Delta_j A_i \| \Delta_{ i+1 }^{ j+1 } \vecf \|_{ \mathbb{R}^n } \right)
	}
	{ A_{ij} \| \Delta_{ i+1 }^{ j+1 } \vecf \|_{ \mathbb{R}^n } }
\]
similarly.
Consequently,
\begin{align*}
	\Delta_i \Delta_j g_{ij,n}
	= & \
	-
	\frac
	{ A_{ij} \| \Delta_{ i+1 }^{ j+1 } \vecf \|_{ \mathbb{R}^n } }
	{ A_{ij} \| \Delta_i^j \vecf \|_{ \mathbb{R}^n } }
	\Delta_i \vecf \cdot \Delta_j \vecf
	-
	\frac
	{ A_{ij} \| \Delta_i^j \vecf \|_{ \mathbb{R}^n } }
	{ A_{ij} \| \Delta_{ i+1 }^{ j+1 } \vecf \|_{ \mathbb{R}^n } }
	\Delta_{i+1} \vecf \cdot \Delta_{j+1} \vecf
	\\
	& \quad
	- \,
	\frac
	{
	\left( A_{ij} \| \Delta_{ i+1 }^{ j+1 } \vecf \|_{ \mathbb{R}^n } \right)
	\left( \Delta_i A_j \| \Delta_i^j \vecf \|_{ \mathbb{R}^n } \right)
	\left( \Delta_j A_i \| \Delta_i^j \vecf \|_{ \mathbb{R}^n } \right)
	}
	{ A_{ij} \| \Delta_i^j \vecf \|_{ \mathbb{R}^n } }
	\\
	& \quad
	- \,
	\frac
	{
	\left( A_{ij} \| \Delta_i^j \vecf \|_{ \mathbb{R}^n } \right)
	\left( \Delta_i A_j \| \Delta_{ i+1 }^{ j+1 } \vecf \|_{ \mathbb{R}^n } \right)
	\left( \Delta_j A_i \| \Delta_{ i+1 }^{ j+1 } \vecf \|_{ \mathbb{R}^n } \right)
	}
	{ A_{ij} \| \Delta_{ i+1 }^{ j+1 } \vecf \|_{ \mathbb{R}^n } }
	\\
	& \quad
	+ \,
	\left( \Delta_j A_i \| \Delta_i^j \vecf \|_{ \mathbb{R}^n } \right)
	\left( \Delta_i A_j \| \Delta_{ i+1 }^{ j+1 } \vecf \|_{ \mathbb{R}^n } \right)
	\\
	& \quad
	+ \,
	\left( \Delta_i A_j \| \Delta_i^j \vecf \|_{ \mathbb{R}^n } \right)
	\left( \Delta_j A_i \| \Delta_{ i+1 }^{ j+1 } \vecf \|_{ \mathbb{R}^n } \right)
	.
\end{align*}
Using this,
we decomposed $ \mathrm{I} $ as
\begin{align*}
	\mathrm{I}
	= & \
	\mathrm{I}_1 + \mathrm{I}_2
	,
	\\
	\mathrm{I}_1
	= & \
	\frac { g_{ij,d} } { g_{ij,n} }
	\left[
	1
	-
	\frac 1
	{ 2
	( H_i \| \Delta_i \vecf \|_{ \mathbb{R}^n } ) ( H_j \| \Delta_j \vecf \|_{ \mathbb{R}^n } )
	}
	\right.
	\\
	& \quad \qquad \qquad
	\left.
	\times
	\left\{
	\frac
	{ A_{ij} \| \Delta_{ i+1 }^{ j+1 } \vecf \|_{ \mathbb{R}^n } }
	{ A_{ij} \| \Delta_i^j \vecf \|_{ \mathbb{R}^n } }
	\left( \Delta_i \vecf \cdot \Delta_j \vecf \right)
	+
	\frac
	{ A_{ij} \| \Delta_i^j \vecf \|_{ \mathbb{R}^n } }
	{ A_{ij} \| \Delta_{ i+1 }^{ j+1 } \vecf \|_{ \mathbb{R}^n } }
	\left( \Delta_{ i+1} \vecf \cdot \Delta_{ j+1 } \vecf \right)
	\right\}
	\right]
	,
	\\
	\mathrm{I}_2
	= & \
	\frac 12
	\frac { g_{ij,d} } { g_{ij,n} }
	\frac
	1
	{ ( H_i \| \Delta_i \vecf \|_{ \mathbb{R}^n } ) ( H_j \| \Delta_j \vecf \|_{ \mathbb{R}^n } ) }
	\\
	& \quad
	\times
	\left\{
	- \,
	\frac
	{
	\left( A_{ij} \| \Delta_{ i+1 }^{ j+1 } \vecf \|_{ \mathbb{R}^n } \right)
	\left( \Delta_i A_j \| \Delta_i^j \vecf \|_{ \mathbb{R}^n } \right)
	\left( \Delta_j A_i \| \Delta_i^j \vecf \|_{ \mathbb{R}^n } \right)
	}
	{ A_{ij} \| \Delta_i^j \vecf \|_{ \mathbb{R}^n } }
	\right.
	\\
	& \quad \qquad
	\left.
	- \,
	\frac
	{
	\left( A_{ij} \| \Delta_i^j \vecf \|_{ \mathbb{R}^n } \right)
	\left( \Delta_i A_j \| \Delta_{ i+1 }^{ j+1 } \vecf \|_{ \mathbb{R}^n } \right)
	\left( \Delta_j A_i \| \Delta_{ i+1 }^{ j+1 } \vecf \|_{ \mathbb{R}^n } \right)
	}
	{ A_{ij} \| \Delta_{ i+1 }^{ j+1 } \vecf \|_{ \mathbb{R}^n } }
	\right.
	\\
	& \quad \qquad
	\left.
	+ \,
	\left( \Delta_j A_i \| \Delta_i^j \vecf \|_{ \mathbb{R}^n } \right)
	\left( \Delta_i A_j \| \Delta_{ i+1 }^{ j+1 } \vecf \|_{ \mathbb{R}^n } \right)
	\right.
	\\
	& \quad \qquad
	\left.
	+ \,
	\left( \Delta_i A_j \| \Delta_i^j \vecf \|_{ \mathbb{R}^n } \right)
	\left( \Delta_j A_i \| \Delta_{ i+1 }^{ j+1 } \vecf \|_{ \mathbb{R}^n } \right)
	\vphantom{
	\frac
	{
	\left( A_{ij} \| \Delta_{ i+1 }^{ j+1 } \vecf \|_{ \mathbb{R}^n } \right)
	\left( A_j \Delta_i \| \Delta_i^j \vecf \|_{ \mathbb{R}^n } \right)
	\left( A_i \Delta_j \| \Delta_i^j \vecf \|_{ \mathbb{R}^n } \right)
	}
	{ A_{ij} \| \Delta_i^j \vecf \|_{ \mathbb{R}^n } }
	}
	\right\}
	.
\end{align*}
Furthermore,
$ \mathrm{I}_1 $ is decomposed into three parts:
\begin{align*}
	\mathrm{I}_1
	= & \
	\mathrm{I}_{11} + \mathrm{I}_{12} + \mathrm{I}_{13}
	,
	\\
	\mathrm{I}_{11}
	= & \
	\frac 12
	\frac { g_{ij,d} } { g_{ij,n} }
	\left\{
	\left(
	1 -
	\frac { \Delta_i \vecf } { \| \Delta_i \vecf \|_{ \mathbb{R}^n } }
	\cdot
	\frac { \Delta_j \vecf } { \| \Delta_j \vecf \|_{ \mathbb{R}^n } }
	\right)
	+
	\left(
	1 -
	\frac { \Delta_{ i+1 } \vecf } { \| \Delta_{ i+1 } \vecf \|_{ \mathbb{R}^n } }
	\cdot
	\frac { \Delta_{ j+1 } \vecf } { \| \Delta_{ j+1 } \vecf \|_{ \mathbb{R}^n } }
	\right)
	\right\}
	,
	\\
	\mathrm{I}_{12}
	= & \
	\frac 12
	\frac { g_{ij,d} } { g_{ij,n} }
	\left\{
	\left(
	1
	-
	\frac
	{ A_{ij} \| \Delta_{ i+1 }^{ j+1 } \vecf \|_{ \mathbb{R}^n } }
	{ A_{ij} \| \Delta_i^j \vecf \|_{ \mathbb{R}^n } }
	\right)
	\frac { \Delta_i \vecf } { \| \Delta_i \vecf \|_{ \mathbb{R}^n } }
	\cdot
	\frac { \Delta_j \vecf } { \| \Delta_j \vecf \|_{ \mathbb{R}^n } }
	\right.
	\\
	& \quad \qquad \qquad
	\left.
	+ \,
	\left(
	1
	-
	\frac
	{ A_{ij} \| \Delta_i^j \vecf \|_{ \mathbb{R}^n } }
	{ A_{ij} \| \Delta_{ i+1 }^{ j+1 } \vecf \|_{ \mathbb{R}^n } }
	\right)
	\frac { \Delta_{ i+1 } \vecf } { \| \Delta_{ i+1} \vecf \|_{ \mathbb{R}^n } }
	\cdot
	\frac { \Delta_{ j+1 } \vecf } { \| \Delta_{ j+1 } \vecf \|_{ \mathbb{R}^n } }
	\right\}
	,
	\\
	\mathrm{I}_{13}
	= & \
	\frac 12
	\frac { g_{ij,d} } { g_{ij,n} }
	\left\{
	\frac 1
	{ \| \Delta_i \vecf \|_{ \mathbb{R}^n } \| \Delta_j \vecf \|_{ \mathbb{R}^n } }
	-
	\frac 1
	{ ( H_i \| \Delta_i \vecf \|_{ \mathbb{R}^n } ) ( H_j \| \Delta_j \vecf \|_{ \mathbb{R}^n } ) }
	\right\}
	\\
	& \quad
	\times
	\left\{
	\frac
	{ A_{ij} \| \Delta_{ i+1 }^{ j+1 } \vecf \|_{ \mathbb{R}^n } }
	{ A_{ij} \| \Delta_i^j \vecf \|_{ \mathbb{R}^n } }
	\left( \Delta_i \vecf \cdot \Delta_j \vecf \right)
	+
	\frac
	{ A_{ij} \| \Delta_i^j \vecf \|_{ \mathbb{R}^n } }
	{ A_{ij} \| \Delta_{ i+1 }^{ j+1 } \vecf \|_{ \mathbb{R}^n } }
	\left( \Delta_{ i+1 } \vecf \cdot \Delta_{ j+1 } \vecf \right)
	\right\}
	.
\end{align*}
Now put
\[
	\mathscr{P}_{1,ij}^m ( \vecf )
	= \mathrm{I}_{11} ,
	\quad
	\mathscr{R}_{1,ij}^m ( \vecf )
	=
	\mathrm{I}_{12} + \mathrm{I}_{13} + \mathrm{I}_2 +
	\mathrm{II} + \mathrm{III} + \mathrm{IV}
	.
\]
Since
\[
	\mathscr{P}_{1,ij}^m ( \vecf )
	=
	\frac 14
	\frac { g_{ij,d} } { g_{ij,n} }
	\left\{
	\left\|
	\frac { \Delta_i \vecf } { \| \Delta_i \vecf \|_{ \mathbb{R}^n } }
	-
	\frac { \Delta_j \vecf } { \| \Delta_j \vecf \|_{ \mathbb{R}^n } }
	\right\|_{ \mathbb{R}^n }^2
	+
	\left\|
	\frac { \Delta_{ i+1 } \vecf } { \| \Delta_{ i+1 } \vecf \|_{ \mathbb{R}^n } }
	-
	\frac { \Delta_{ j+1 } \vecf } { \| \Delta_{ j+1 } \vecf \|_{ \mathbb{R}^n } }
	\right\|_{ \mathbb{R}^n }^2
	\right\}
	\geqq 0
	,
\]
we have
\[
	\mathscr{P}_{1,ij}^m ( \vecf )
	\to
	\frac 12
	\frac { \| \Delta_i^j \vectau \|_{ \mathbb{R}^n }^2 }
	{ \| \Delta_i^j \vecf \|_{ \mathbb{R}^n }^2 }
	\| \dot { \vecf } ( \theta_i ) \|_{ \mathbb{R}^n }
	\| \dot { \vecf } ( \theta_j ) \|_{ \mathbb{R}^n }
	d \theta_i d \theta_j
\]
as $ m \to \infty $ formally,
which is the energy density of $ \mathcal{E}_1 ( \vecf ) $ at $ ( \theta_i , \theta_j ) $.
Similarly $ \mathscr{R}_{1,ij}^m ( \vecf ) \to 0 $ as $ m \to \infty $ formally.
Hence we expect \pref{convergence_of_the_discrete_1st_energy}.
\par
Similarly $ \mathscr{M}_{ij}^m $ is decomposed as
\[
	\mathscr{M}_{ij}^m ( \vecf )
	=
	\mathrm{J}_1
	+
	\mathrm{J}_2
	+
	\mathrm{J}_3
	+
	\mathrm{J}_4
	+
	\mathrm{J}_5
	,
\]
where
\begin{align*}
	\mathrm{J}_1
	= & \
	\frac { g_{ij,d} } { g_{ij,n} }
	\left[
	1
	-
	\frac 1
	{ 2 g_{ij,d} }
	\left\{
	\Delta_i \Delta_j g_{ij,n}
	-
	\frac {
	\left( \Delta_i g_{ij,n} \right)
	\left( \Delta_j g_{ij,n} \right)
	}
	{ g_{ij,n} }
	\right\}
	- \frac { g_{ij,d} } { 2 g_{ij,n} }
	\right]
	,
	\\
	\\
	\mathrm{J}_2
	= & \
	\frac { g_{ij,d} } { 2 g_{ij,n} }
	\left\{
	\frac 1 { g_{ij,d} }
	-
	\frac 1
	{ ( H_i \| \Delta_i \vecf \|_{ \mathbb{R}^n } ) ( H_j \| \Delta_j \vecf \|_{ \mathbb{R}^n } ) }
	\right\}
	\left\{
	\Delta_i \Delta_j g_{ij,n}
	-
	\frac {
	\left( \Delta_i g_{ij,n} \right)
	\left( \Delta_j g_{ij,n} \right)
	}
	{ g_{ij,n} }
	\right\}
	,
	\\
	\mathrm{J}_3
	= & \
	-
	\frac { g_{ij,d} } { 2 g_{ij,n} H_j \| \Delta_i \vecf \|_{ \mathbb{R}^n } }
	\left\{
	\left( \Delta_j A_i g_{ij,n} \right)
	\left( \Delta_i \frac 1 { H_i \| \Delta_i \vecf \|_{ \mathbb{R}^n } } \right)
	-
	\frac { A_i g_{ij,n} } { g_{ij,n} }
	\left( \Delta_j g_{ij,n} \right)
	\left( \Delta_i \frac 1 { \| \Delta_i \vecf \|_{ \mathbb{R}^n } } \right)
	\right\}
	,
	\\
	\mathrm{J}_4
	= & \
	-
	\frac { g_{ij,d} } { 2 g_{ij,n} H_i \| \Delta_i \vecf \|_{ \mathbb{R}^n } }
	\left\{
	\left( \Delta_i A_j g_{ij,n} \right)
	\left( \Delta_j \frac 1 { H_j \| \Delta_i \vecf \|_{ \mathbb{R}^n } } \right)
	-
	\frac { A_j g_{ij,n} } { g_{ij,n} }
	\left( \Delta_i g_{ij,n} \right)
	\left( \Delta_j \frac 1 { \| \Delta_i \vecf \|_{ \mathbb{R}^n } } \right)
	\right\}
	,
	\\
	\mathrm{J}_5
	= & \
	-
	\frac { g_{ij,d} } { 2 g_{ij,n} }
	\left\{
	A_{ij} g_{ij,n}
	-
	\frac {
	\left( A_i g_{ij,n} \right)
	\left( A_j g_{ij,n} \right)
	}
	{ g_{ij,n} }
	\right\}
	\left( \Delta_i \frac 1 { \| \Delta_i \vecf \|_{ \mathbb{R}^n } } \right)
	\left( \Delta_j \frac 1 { \| \Delta_j \vecf \|_{ \mathbb{R}^n } } \right)
	.
\end{align*}
Futhermore we decomposed $ \mathrm{J}_1 $ as
\begin{align*}
	\mathrm{J}_1 = & \
	\mathrm{J}_{11} + \mathrm{J}_{12} + \mathrm{J}_{13} + \mathrm{J}_{14}
	,
	\\
	\mathrm{J}_{11}
	= & \
	\frac { g_{ij,d} } { \bar g_{ij,n} }
	\left[
	1
	-
	\frac 
	1
	{ 2 g_{ij,d} }
	\left\{
	\Delta_i \Delta_j \bar g_{ij,n}
	+
	\frac { 4 \left( \Delta_i^j \vecf \cdot \Delta_i \vecf \right) \left( \Delta_i^j \vecf \cdot \Delta_j \vecf \right) }
	{ \bar g_{ij,n} }
	\right\}
	\right]
	,
	\\
	\mathrm{J}_{12}
	= & \
	g_{ij,d}
	\left(
	\frac 1 { g_{ij,n} } - \frac 1 { \bar g_{ij,n} }
	\right)
	,
	\\
	\mathrm{J}_{13}
	= & \
	-
	\frac 1 { 2 g_{ij,n} } \Delta_i \Delta_j g_{ij,n}
	+
	\frac 1 { 2 \bar g_{ij,n} } \Delta_i \Delta_j \bar g_{ij,n}
	,
	\\
	\mathrm{J}_{14}
	= & \
	\frac 12
	\left\{
	\frac { \left( \Delta_i g_{ij,n} \right) \left( \Delta_j g_{ij,n} \right) }
	{ g_{ij,n}^2 }
	+
	\frac { 4 \left( \Delta_i^j \vecf \cdot \Delta_i \vecf \right) \left( \Delta_i^j \vecf \cdot \Delta_j \vecf \right) }
	{ \bar g_{ij,n}^2 }
	- \left( \frac { g_{ij,d} } { g_{ij,n}^2 } \right)^2
	\right\}
	.
\end{align*}
We put
\[
	\mathscr{P}_{ij}^m ( \vecf ) = \mathrm{J}_{11} ,
	\quad
	\mathscr{R}_{ij}^m ( \vecf )
	=
	\mathrm{J}_{12} + \mathrm{J}_{13} + \mathrm{J}_{14} + \mathrm{J}_2 + \mathrm{J}_3 + \mathrm{J}_4 + \mathrm{J}_5 .
\]
Since $ J_{11} $ can be rewritten as
\begin{align}
	\mathrm{J}_{11}
	= & \
	\frac 12
	\frac { g_{ij,d} } { \bar g_{ij,n} }
	\left[
	\left\|
	\frac { \Delta_i^j \vecf } { \| \Delta_i^j \vecf \|_{ \mathbb{R}^n } }
	\bigwedge
	\left(
	\frac { \Delta_i \vecf } { \| \Delta_i \vecf \|_{ \mathbb{R}^n } }
	+
	\frac { \Delta_j \vecf } { \| \Delta_j \vecf \|_{ \mathbb{R}^n } }
	\right)
	\right\|_{ \wedge^2 \mathbb{R}^n }^2
	\right.
	\label{J_11}
	\\
	& \quad \qquad
	\left.
	+ \,
	\left\{
	\frac { \Delta_i^j \vecf } { \| \Delta_i^j \vecf \|_{ \mathbb{R}^n } }
	\cdot
	\left(
	\frac { \Delta_i \vecf } { \| \Delta_i \vecf \|_{ \mathbb{R}^n } }
	-
	\frac { \Delta_j \vecf } { \| \Delta_j \vecf \|_{ \mathbb{R}^n } }
	\right)
	\right\}^2
	\right]
	,
	\nonumber
\end{align}
it formally holds that
\begin{align*}
	\mathscr{P}_{ij}^m ( \vecf )
	\to & \
	\frac 12
	\left\{
	\left\|
	\frac { \Delta_i^j \vecf } { \| \Delta_i^j \vecf \|_{ \mathbb{R}^n } }
	\bigwedge
	( \vectau ( \theta_i ) + \vectau ( \theta_j ) )
	\right\|_{ \wedge^2 \mathbb{R}^n }^2
	+
	\left(
	\frac { \Delta_i^j \vecf } { \| \Delta_i^j \vecf \|_{ \mathbb{R}^n } }
	\cdot
	\Delta_i^j \vectau
	\right)^2
	\right\}
	\\
	& \quad \qquad
	\times
	\frac { \| \dot { \vecf } ( \theta_i ) \|_{ \mathbb{R}^n } \| \dot { \vecf }( \theta_j ) \|_{ \mathbb{R}^n } }
	{ \| \Delta_i^j \vecf \|_{ \mathbb{R}^n }^2 }
	\,
	d \theta_i d \theta_j
\end{align*}
as $ m \to \infty $,
which is the energy density of $ \mathcal{E} ( \cdot ) - 4 $ at $ ( \theta_i , \theta_j ) $.
On the other hand,
we can expect $ \mathscr{R}_{ij}^m ( \vecf ) \to 0 $ as $ m \to \infty $.
\subsection{Estimates}
\par
We would like to show the convergence of our discrete energies.
In this subsection,
we will derive several estimates for proving the convergence.
It is known that $ \vecf $ is bi-Lipschitz if $ \mathcal{E} ( \vecf ) < \infty $.
Therefore it is natural to assume the bi-Lipschitz continuity of $ \vecf $.
Then there exist positive constant $ L_1 $ and $ L_2 $ independent of
$ i $,
$ j $,
and
$ m $ such that
\be
	L_1 d_{\mathbb{R}/\mathbb{Z}} ( \theta_i , \theta_j )
	\leqq
	\| \Delta_i^j \vecf \|_{ \mathbb{R}^n }
	\leqq
	L_2 d_{\mathbb{R}/\mathbb{Z}} ( \theta_i , \theta_j )
	\label{bi-Lipschitz}
	.
\ee
Here
$ d_{\mathbb{R}/\mathbb{Z}} ( \cdot , \cdot ) $ is the distance of $ \mathbb{R}/\mathbb{Z} $.
\par
Furthermore we assume the equi-lateral condition,
{\it i.e.},
\be
	\| \Delta_i \vecf \|_{ \mathbb{R}^n } = \frac {L_m } m
	\quad ( i = 1 , \ \cdots , m )
	.
	\label{equilateral}
\ee
Here $ L_m $ is the total length of $ m $-polygon with the $ i $-th vertex $ \vecf ( \theta_i ) $.
\[
	L_m \leqq L
\]
is obvious.
We may assume that $ m $ is sufficiently large so that
\[
	d_{ \mathbb{R}/\mathbb{Z} } ( \theta_i , \theta_{ i+1 } ) = \Delta_i \theta .
\]
It follows from \pref{bi-Lipschitz} and \pref{equilateral} that
\[
	L_1 \Delta_i \theta
	\leqq
	\| \Delta_i \vecf \|_{ \mathbb{R}^n }
	=
	\frac { L_m } m
	\leqq
	L_2 \Delta_i \theta
	,
\]
and hence we have
\[
	\frac { L_m } { L_2 m } \leqq \Delta_i \theta \leqq \frac { L_m } { L_1 m }
	\leqq \frac L { L_1 m } .
\]
In particular,
\[
	\Delta_i \theta \to 0 \quad \mbox{as } m \to \infty
\]
holds.
we can estimate $ \Delta_i \theta $ from below by $ \frac C m $.
To show this,
it it sufficient to see $ L_m \to L $ as $ m \to \infty $.
\par
Let $ I $,
$ J \subset \mathbb{R} / \mathbb{Z} $ be measurable sets, not necessarily intervals, and we put
\[
	[u]_{ H^{ \frac 12 } ( I \times J ) }
	=
	\left(
	\iint_{ I \times J } \frac { | \Delta_i^j u |^2 } { | \Delta_i^j \theta |^2 } d \theta_i d \theta_j
	\right)^{ \frac 12 }
\]
for $ H^{ \frac 12 } ( I \times J ) $.
When we consider an $ \mathbb{R}^n $-valued function,
we define
\[
	[ \vecu ]_{ H^{ \frac 12 } ( I \times J ) }
	=
	\left(
	\iint_{ I \times J } \frac { \| \Delta_i^j \vecu \|_{ \mathbb{R}^n }^2 } { | \Delta_i^j \theta |^2 } d \theta_i d \theta_j
	\right)^{ \frac 12 } .
\]
Put
$ I_i = [ \theta_i , \theta_{ i+1 } ] $,
$ I_j = [ \theta_j , \theta_{ j+1 } ] $.
\begin{lem}
It holds for $ \vecf \in H^{ \frac 32 } ( \mathbb{R}/ \mathbb{Z} ) $ that
\[
	0 \leqq L - L_m
	\leqq
	\frac { 2 L } { \sqrt 6 \, m }
	[ \dot { \vecf } ]_{ H^{ \frac 12 } ( \mathbb{R}/ \mathbb{Z} ) }
	.
\]
In particular
\be
	\frac L2 \leqq L_m
	\label{lower_bound_of_Lm}
\ee
holds for sufficiently large $ m $.
\end{lem}
\begin{proof}
Let $ \theta_\ast \in [ \theta_i , \theta_{ i+1 } ] $.
Then we have
\begin{align*}
	&
	\int_{ \theta_i }^{ \theta_{ i+1 } }
	\left\| \dot { \vecf } ( \theta ) \right\|_{ \mathbb{R}^n } d \theta
	-
	\left\|
	\int_{ \theta_i }^{ \theta_{ i+1 } } \dot { \vecf } ( \theta ) \, d \theta
	\right\|_{ \mathbb{R}^n }
	\\
	& \quad
	\leqq
	\int_{ \theta_i }^{ \theta_{ i+1 } }
	\left\| \dot { \vecf } ( \theta ) - \dot { \vecf } ( \theta_\ast ) \right\|_{ \mathbb{R}^n } d \theta
	+
	\left\| \int_{ \theta_i }^{ \theta_{ i+1 } }
	\left( \dot { \vecf } ( \theta ) - \dot { \vecf } ( \theta_\ast ) \right) d \theta
	\right\|_{ \mathbb{R}^n }
	\\
	& \quad
	\leqq
	2
	\int_{ \theta_i }^{ \theta_{ i+1 } }
	\left\| \dot { \vecf } ( \theta ) - \dot { \vecf } ( \theta_\ast ) \right\|_{ \mathbb{R}^n } d \theta
	.
\end{align*}
We integrate this with respect to $ \theta_\ast $ on $ I_i $,
and divide the result by $ \Delta_i \theta $ to get
\begin{align*}
	&
	\int_{ \theta_i }^{ \theta_{ i+1 } }
	\left\| \dot { \vecf } ( \theta ) \right\|_{ \mathbb{R}^n } d \theta
	-
	\left\|
	\int_{ \theta_i }^{ \theta_{ i+1 } } \dot { \vecf } ( \theta ) \, d \theta
	\right\|_{ \mathbb{R}^n }
	\\
	& \quad
	\leqq
	\frac 2 { \Delta_i \theta }
	\int_{ \theta_i }^{ \theta_{ i+1 } }
	\int_{ \theta_i }^{ \theta_{ i+1 } }
	\left\| \dot { \vecf } ( \theta ) - \dot { \vecf } ( \theta_\ast ) \right\|_{ \mathbb{R}^n }
	d \theta d \theta_\ast
	\\
	& \quad
	\leqq
	\frac 2 { \Delta_i \theta }
	\left(
	\int_{ \theta_i }^{ \theta_{ i+1 } }
	\int_{ \theta_i }^{ \theta_{ i+1 } }
	| \theta - \theta_\ast |^2
	d \theta d \theta_\ast
	\right)^{ \frac 12 }
	\left(
	\int_{ \theta_i }^{ \theta_{ i+1 } }
	\int_{ \theta_i }^{ \theta_{ i+1 } }
	\frac {
	\left\| \dot { \vecf } ( \theta ) - \dot { \vecf } ( \theta_\ast ) \right\|_{ \mathbb{R}^n }^2
	}
	{ | \theta - \theta_\ast |^2 }
	d \theta d \theta_\ast
	\right)^{ \frac 12 }
	\\
	& \quad
	\leqq
	\frac { 2 \Delta_i \theta } { \sqrt 6 }
	[ \dot { \vecf } ]_{ H^{ \frac 12 } ( I_i \times I_i ) }
	.
\end{align*}
Summing with respect to $ i $,
we obtain
\begin{align*}
	0 
	\leqq & \
	L - L_m
	\\
	\leqq & \
	\sum_{ i=1 }^m
	\left(
	\int_{ \theta_i }^{ \theta_{ i+1 } }
	\left\| \dot { \vecf } ( \theta ) \right\|_{ \mathbb{R}^n } d \theta
	-
	\left\|
	\int_{ \theta_i }^{ \theta_{ i+1 } } \dot { \vecf } ( \theta ) \, d \theta
	\right\|_{ \mathbb{R}^n }
	\right)
	\leqq 
	\frac { 2 L } { \sqrt 6 \, m }
	[ \dot { \vecf } ]_{ H^{ \frac 12 } ( \mathbb{R}/ \mathbb{Z} ) }
	.
\end{align*}
\qed
\end{proof}
\par
By this lemma,
we may assume
\be
	\frac L { 2 L_1 m } \leqq \Delta_i \theta \leqq \frac L { L_2 m }
	\label{estimete_Delta_theta}
\ee
for sufficiently large $ m $.
\begin{lem}
Under {\rm \pref{equilateral}},
\[
	\| \Delta_i \vecf - \Delta_{ i+1 } \vecf \|_{ \mathbb{R}^n }
	\leqq
	2 \Delta_i^{ i+2 } \theta [ \dot { \vecf } ]_{ H^{ \frac 12 } ( [ \theta_i , \theta_{ i+2 } ]^2 ) }
\]
holds for $ \vecf \in H^{ \frac 32 } ( \mathbb{R} / \mathbb{Z} ) $.
\label{estimate_Delta(i+1)-Deltai}
\end{lem}
\begin{proof}
Using \pref{equilateral},
we have
\begin{align*}
	&
	\| \Delta_i \vecf - \Delta_{ i+1 } \vecf \|_{ \mathbb{R}^n }
	\\
	& \quad
	= \
	\frac { L_m } m
	\left\|
	\frac { \Delta_i \vecf } { \| \Delta_i \vecf \|_{ \mathbb{R}^n } }
	-
	\frac { \Delta_{ i+1 } \vecf } { \| \Delta_{ i+1 } \vecf \|_{ \mathbb{R}^n } }
	\right\|_{ \mathbb{R}^n }
	\\
	& \quad
	\leqq
	\frac { L_m } m
	\left\{
	\left\|
	\frac { \Delta_i \vecf } { \| \Delta_i \vecf \|_{ \mathbb{R}^n } }
	-
	\vectau ( \theta_\ast )
	\right\|_{ \mathbb{R}^n }
	+
	\left\|
	\vectau ( \theta_\ast )
	-
	\frac { \Delta_{ i+1 } \vecf } { \| \Delta_{ i+1 } \vecf \|_{ \mathbb{R}^n } }
	\right\|_{ \mathbb{R}^n }
	\right\}
	.
\end{align*}
Here we choose $ \theta_\ast $ from the interval $ [ \theta_i , \theta_{ i+2 } ] $.
We integrate this with respect to $ \theta_\ast $,
and divide the resulte by $ \Delta_i^{ i+2 } \theta $ to get
\begin{align*}
	&
	\| \Delta_i \vecf - \Delta_{ i+1 } \vecf \|_{ \mathbb{R}^n }
	\\
	& \quad
	\leqq
	\frac { L_m } m
	\frac 1 { \Delta_i^{ i+2 } \theta }
	\int_{ \theta_i }^{ \theta_{ i+2 } }
	\left(
	\left\|
	\frac { \Delta_i \vecf } { \| \Delta_i \vecf \|_{ \mathbb{R}^n } }
	-
	\vectau ( \theta_\ast )
	\right\|_{ \mathbb{R}^n }
	+
	\left\|
	\vectau ( \theta_\ast )
	-
	\frac { \Delta_{ i+1 } \vecf } { \| \Delta_{ i+1 } \vecf \|_{ \mathbb{R}^n } }
	\right\|_{ \mathbb{R}^n }
	\right)
	d \theta_\ast
	.
\end{align*}
It is clear that
\begin{align*}
	&
	\frac { \Delta_i \vecf } { \| \Delta_i \vecf \|_{ \mathbb{R}^n } }
	-
	\vectau ( \theta_\ast )
	\\
	& \quad
	=
	\frac 1  { \| \Delta_i \vecf \|_{ \mathbb{R}^n } }
	\int_{ \theta_i }^{ \theta_{ i+1 } }
	\left( \dot { \vecf } ( \theta ) - \dot { \vecf } ( \theta_\ast ) \right) d \theta
	+
	\frac { \Delta_i \theta \dot { \vecf } ( \theta_\ast ) }
	{ \| \Delta_i \vecf \|_{ \mathbb{R} } \| \dot { \vecf } ( \theta_\ast ) \|_{ \mathbb{R}^n } }
	\left( 
	\| \dot { \vecf } ( \theta_\ast ) \|_{ \mathbb{R}^n }
	- \frac { \| \Delta_i \vecf \|_{ \mathbb{R}^n } } { \Delta_i \theta }
	\right)
	.
\end{align*}
We use \pref{equilateral} again,
and get
\[
	\left\|
	\frac 1  { \| \Delta_i \vecf \|_{ \mathbb{R}^n } }
	\int_{ \theta_i }^{ \theta_{ i+1 } }
	\left( \dot { \vecf } ( \theta ) - \dot { \vecf } ( \theta_\ast ) \right) d \theta
	\right\|_{ \mathbb{R}^n }
	\leqq
	\frac m { L_m }
	\int_{ \theta_i }^{ \theta_{ i+1 } }
	\left\| \dot { \vecf } ( \theta ) - \dot { \vecf } ( \theta_\ast ) \right\|_{ \mathbb{R}^n }
	d \theta
	.
\]
On the other hand,
it holds that
\begin{align*}
	&
	\left\|
	\frac { \Delta_i \theta \dot { \vecf } ( \theta_\ast ) }
	{ \| \Delta_i \vecf \|_{ \mathbb{R} } \| \dot { \vecf } ( \theta_\ast ) \|_{ \mathbb{R}^n } }
	\left(
	\| \dot { \vecf } ( \theta_\ast ) \|_{ \mathbb{R}^n }
	- \frac { \| \Delta_i \vecf \|_{ \mathbb{R}^n } } { \Delta_i \theta }
	\right)
	\right\|_{ \mathbb{R}^n }
	\\
	& \quad
	\leqq
	\frac m { L_m }
	\left\|
	\Delta \theta_i \dot { \vecf } ( \theta_\ast )
	- \Delta_i \vecf
	\right\|_{ \mathbb{R}^n }
	\\
	& \quad
	\leqq
	\frac m { L_m }
	\int_{ \theta_i }^{ \theta_{ i+1 } }
	\left\| \dot { \vecf } ( \theta ) - \dot { \vecf } ( \theta_\ast ) \right\|_{ \mathbb{R}^n }
	d \theta
	.
\end{align*}
Consequently we obtain
\begin{align*}
	&
	\frac { L_m } m
	\frac 1 { \Delta_i^{ i+2 } \theta }
	\int_{ \theta_i }^{ \theta_{ i+2 } }
	\left(
	\left\|
	\frac { \Delta_i \vecf } { \| \Delta_i \vecf \|_{ \mathbb{R}^n } }
	-
	\vectau ( \theta_\ast )
	\right\|_{ \mathbb{R}^n }
	+
	\left\|
	\vectau ( \theta_\ast )
	-
	\frac { \Delta_{ i+1 } \vecf } { \| \Delta_{ i+1 } \vecf \|_{ \mathbb{R}^n } }
	\right\|_{ \mathbb{R}^n }
	\right)
	d \theta_\dagger
	\\
	& \quad
	\leqq
	\frac 2 { \Delta_i^{ i+2 } \theta }
	\left\{
	\int_{ \theta_i }^{ \theta_{ i+1 } }
	\int_{ \theta_i }^{ \theta_{ i+2 } }
	\left\| \dot { \vecf } ( \theta ) - \dot { \vecf } ( \theta_\ast ) \right\|_{ \mathbb{R}^n }
	d \theta d \theta_\ast
	+
	\int_{ \theta_{ i+1 } }^{ \theta_{ i+2 } }
	\int_{ \theta_i }^{ \theta_{ i+2 } }
	\left\| \dot { \vecf } ( \theta ) - \dot { \vecf } ( \theta_\ast ) \right\|_{ \mathbb{R}^n }
	d \theta d \theta_\ast
	\right\}
	\\
	& \quad
	\leqq
	2
	\int_{ \theta_i }^{ \theta_{ i+2 } }
	\int_{ \theta_i }^{ \theta_{ i+2 } }
	\frac {
	\left\| \dot { \vecf } ( \theta ) - \dot { \vecf } ( \theta_\ast ) \right\|_{ \mathbb{R}^n }
	}
	{ | \theta - \theta_\ast | }
	d \theta d \theta_\ast
	\\
	& \quad
	\leqq
	2
	\left(
	\int_{ \theta_i }^{ \theta_{ i+2 } }
	\int_{ \theta_i }^{ \theta_{ i+2 } }
	d \theta d \theta_\ast
	\right)^{ \frac 12 }
	\left(
	\int_{ \theta_i }^{ \theta_{ i+2 } }
	\int_{ \theta_i }^{ \theta_{ i+2 } }
	\frac {
	\left\| \dot { \vecf } ( \theta ) - \dot { \vecf } ( \theta_\ast ) \right\|_{ \mathbb{R}^n }^2
	}
	{ | \theta - \theta_\ast |^2 }
	d \theta d \theta_\ast
	\right)^{ \frac 12 }
	\\
	& \quad
	\leqq
	2 \Delta_i^{ i+2 } \theta [ \dot { \vecf } ]_{ H^{ \frac 12 } ( [ \theta_i , \theta_{ i+2 } ] ) }
	.
\end{align*}
\qed
\end{proof}
\begin{defi}
For $ u \in H^{ \frac 12 } ( \mathbb{R} / \mathbb{Z} ) $ and positive constants $ \epsilon_1 $,
$ \epsilon_2 $,
we define $ K ( \vecu , \epsilon_1 , \epsilon_2 ) $ by
\[
	K ( \vecu , \epsilon_1 , \epsilon_2 )
	=
	\sup \left\{
	\left.
	[ \vecu ]_{ H^{ \frac 12 } ( I \times J ) } \, \right| \,
	I \subset \mathbb{R} / \mathbb{Z},
	\,
	J \subset \mathbb{R} / \mathbb{Z},
	\,
	| I | \leqq \epsilon_1 , \, | J | \leqq \epsilon_2
	\right\}
	.
\]
When $ \epsilon_1 = \epsilon_2 $,
we will write the quantity simply by $ K( \vecu , \epsilon_1 ) $.
\end{defi}
\begin{cor}
Let $ \vecf \in H^{ \frac 32 } ( \mathbb{R} / \mathbb{Z} ) $ satisfy {\rm \pref{bi-Lipschitz}} and {\rm \pref{equilateral}}.
Then we have
\begin{align*}
	\| \Delta_i \vecf - \Delta_j \vecf \|_{ \mathbb{R}^n }
	\leqq & \
	\frac { CL | i - j | } m K \left( \dot { \vecf } , \frac { 2L } { L_2 m } \right)
	,
	\\
	\left\|
	\frac { \Delta_i \vecf } { \| \Delta_i \vecf \|_{ \mathbb{R}^n } }
	-
	\frac { \Delta_j \vecf } { \| \Delta_j \vecf \|_{ \mathbb{R}^n } }
	\right\|_{ \mathbb{R}^n }
	\leqq & \
	C | i - j | K \left( \dot { \vecf } , \frac { 2L } { L_2 m } \right) .
\end{align*}
\label{cor1}
\end{cor}
\begin{proof}
These estimates are shown as
\begin{align*}
	\| \Delta_i \vecf - \Delta_j \vecf \|_{ \mathbb{R}^n }
	\leqq & \
	\sum_{ k=i }^{ j-1 }
	\| \Delta_k \vecf - \Delta \vecf_{ k+1 } \|_{ \mathbb{R}^n }
	\\
	\leqq & \
	\sum_{ k=1 }^{ j-1 }
	\Delta_k^{ k+2 } \theta
	[ \dot { \vecf } ]_{ H^{ \frac 12 } ( [ \theta_k , \theta_{ k+2 } ]^2 ) }
	\\
	\leqq & \
	\frac { CL | i - j | } m K \left( \dot { \vecf } , \frac { 2L } { L_2 m } \right)
	,
	\\
	\left\|
	\frac { \Delta_i \vecf } { \| \Delta_i \vecf \|_{ \mathbb{R}^n } }
	-
	\frac { \Delta_j \vecf } { \| \Delta_j \vecf \|_{ \mathbb{R}^n } }
	\right\|_{ \mathbb{R}^n }
	= & \
	\frac m { L_m } \| \Delta_i \vecf - \Delta \vecf_j \|_{ \mathbb{R}^n }
	\\
	\leqq & \
	C | i - j | K \left( \dot { \vecf } , \frac { 2L } { L_2 m } \right)
	.
\end{align*}
\qed
\end{proof}
\par
It follows from \pref{bi-Lipschitz} and \pref{equilateral} that
\[
	L_1 \Delta_k \theta \leqq \| \Delta_k \vecf \|_{ \mathbb{R}^n }
	= \frac { L_m } m
	\leqq
	L_2 \Delta_k \theta \leqq \| \Delta_k \vecf \|_{ \mathbb{R}^n }
	.
\]
When $ d_{ \mathbb{R}/\mathbb{Z} } ( \theta_i , \theta_j ) = \theta_j - \theta_i $,
we sum the above estimate with respect to $ k $ from $ i $ to $ j - 1 $.
The result is
\[
	L_1 d_{ \mathbb{R}/\mathbb{Z} } ( \theta_i , \theta_j )
	\leqq
	\frac { j - i } m L_m
	\leqq
	L_2 d_{ \mathbb{R}/\mathbb{Z} } ( \theta_i , \theta_j )
	.
\]
From \pref{bi-Lipschitz} we have
\[
	\frac { L_m L_1 } { L_2 }
	\frac { j - i } m
	\leqq
	L_1 d_{ \mathbb{R}/\mathbb{Z} } ( \theta_i , \theta_j )
	\leqq
	\| \Delta_i^j \vecf \|_{ \mathbb{R}^n }
	\leqq
	L_2 d_{ \mathbb{R}/\mathbb{Z} } ( \theta_i , \theta_j )
	\leqq
	\frac { L_m L_2 } { L_1 }
	\frac { j - i } m
	.
\]
When $ d_{ \mathbb{R}/\mathbb{Z} } ( \theta_i , \theta_j ) = \theta_{ i+m } - \theta_j $,
we have similarly
\[
	L_1 d_{ \mathbb{R}/\mathbb{Z} } ( \theta_i , \theta_j )
	\leqq
	\frac { m + i - j } m L_m
	\leqq
	L_2 d_{ \mathbb{R}/\mathbb{Z} } ( \theta_i , \theta_j )
	,
\]
and
\[
	\frac { L_m L_1 } { L_2 }
	\frac { m + i - j } m
	\leqq
	L_1 d_{ \mathbb{R}/\mathbb{Z} } ( \theta_i , \theta_j )
	\leqq
	\| \Delta_i^j \vecf \|_{ \mathbb{R}^n }
	\leqq
	L_2 d_{ \mathbb{R}/\mathbb{Z} } ( \theta_i , \theta_j )
	\leqq
	\frac { L_m L_2 } { L_1 }
	\frac { m + i - j } m
	.
\]
Hence we find
\[
	\frac { L L_1 } { 2 L_2 }
	\frac { \min\{ | i - j | , m - | i - j | \} } m
	\leqq
	\| \Delta_i^j \vecf \|_{ \mathbb{R}^n }
	\leqq
	\frac { L L_2 } { L_1 }
	\frac { \max \{ | i - j | , m - | i - j | \} } m
\]
for sufficnetly large $ m $.
We consider the summation with respect to $ i $ and $ j $
\[
	\sum_{ i \ne j } ( \cdots )
	=
	\sum_{ i=1 }^m \left( \sum_{ j = i - [ \frac m2 ] }^{ i - 1 } + \sum_{ j = i + 1 }^{ i + [ \frac m2 ] } \right) ( \cdots )
\]
when $ m $ is odd;
and
\[
	\sum_{ i \ne j } ( \cdots )
	=
	\sum_{ i=1 }^m \left( \sum_{ j = i - [ \frac m2 ] +1 }^{ i - 1 } + \sum_{ j = i + 1 }^{ i + [ \frac m2 ] } \right) ( \cdots )
\]
when $ m $ is even.
Thus we may assume $ | i - j | \leqq [ \frac m2 ] $.
\begin{cor}
There exist constants $ C $ depending on
$ L_1 $ and 
$ L_2 $ such that
\begin{align*}
	\left| A_{ij} \left(
	\| \Delta_{ i+1 }^{ j+1 } \vecf \|_{ \mathbb{R}^n }
	-
	\| \Delta_i^j \vecf \|_{ \mathbb{R}^n }
	\right) \right|
	\leqq & \
	A_{ij}
	\| \Delta_{ i+1 }^{ j+1 } \vecf - \Delta_i^j \vecf \|_{ \mathbb{R}^n }
	\\
	\leqq & \
	\frac { C L | i - j | } m
	K \left( \dot { \vecf } , \frac { 2 L } { L_2 m } \right)
	.
\end{align*}
\label{cor2}
\end{cor}
\begin{proof}
From Corollary \ref{cor1} we have
\begin{align*}
	\left|
	\| \Delta_{ i+1 }^{ j+1 } \vecf \|_{ \mathbb{R}^n }
	-
	\| \Delta_i^j \vecf \|_{ \mathbb{R}^n }
	\right|
	\leqq & \
	\| \Delta_{ i+1 }^{ j+1 } \vecf - \Delta_i^j \vecf \|_{ \mathbb{R}^n }
	=
	\| \Delta_j \vecf - \Delta_i \vecf \|_{ \mathbb{R}^n }
	\\
	\leqq & \
	\frac { CL | i - j |} m
	K \left( \dot { \vecf } , \frac { 2 L } { L_2 m } \right)
	.
\end{align*}
Similarly we obtain
\begin{align*}
	\left|
	\| \Delta_{ i+1 }^{ j+2 } \vecf \|_{ \mathbb{R}^n }
	-
	\| \Delta_i^{ j+1 } \vecf \|_{ \mathbb{R}^n }
	\right|
	\leqq & \
	\frac { CL | i - j - 1 |} m
	K \left( \dot { \vecf } , \frac { 2 L } { L_2 m } \right)
	,
	\\
	\left|
	\| \Delta_{ i+2 }^{ j+1 } \vecf \|_{ \mathbb{R}^n }
	-
	\| \Delta_{ i+1 }^j \vecf \|_{ \mathbb{R}^n }
	\right|
	\leqq & \
	\frac { CL | i - j + 1 |} m
	K \left( \dot { \vecf } , \frac { 2 L } { L_2 m } \right)
	,
	\\
	\left|
	\| \Delta_{ i+2 }^{ j+2 } \vecf \|_{ \mathbb{R}^n }
	-
	\| \Delta_{ i+1 }^{ j +1 } \vecf \|_{ \mathbb{R}^n }
	\right|
	\leqq & \
	\frac { CL | i - j |} m
	K \left( \dot { \vecf } , \frac { 2 L } { L_2 m } \right)
	.
\end{align*}
The assertion follows by calculating the arithmetic mean of these.
\qed
\end{proof}
\begin{lem}
If $ m $ is sufficiently larga,
and if $ | i - j | \leqq [ \frac m2 ] $,
then
\begin{align*}
	A_i \| \Delta_i^j \vecf \|_{ \mathbb{R}^n } \geqq & \
	\frac { L_m L_1 } { L_2 }
	\frac { | i - j | - \frac 12} m
	,
	\\
	A_i \| \Delta_i^{ j+1 } \vecf \|_{ \mathbb{R}^n } \geqq & \
	\frac { L_m L_1 } { L_2 }
	\frac { | i - j | - \frac 12} m
	,
	\\
	A_{ij} \| \Delta_i^j \vecf \|_{ \mathbb{R}^n } \geqq & \
	\frac { L_m L_1 } { L_2 }
	\frac { | i - j | - \frac 14} m
	,
	\\
	A_{ij} \| \Delta_i^{ j+1 } \vecf \|_{ \mathbb{R}^n } \geqq & \
	\frac { L_m L_1 } { 2 L_2 }
	\frac { | i - j | } m
	.
\end{align*}
\label{lem3}
\end{lem}
\begin{proof}
When $ i < j < i + [ \frac m2 ] - 1 $,
we have
\begin{align*}
	0 < j - i < \left[ \frac m2 \right] - 1 ,
	& \quad
	0 < j + 1 - i \leqq \left[ \frac m2 \right] - 1 ,
	\\
	0 \leqq j - i - 1 < \left[ \frac m2 \right] - 1 ,
	& \quad
	0 < ( j + 1 ) - ( i + 1) < \left[ \frac m2 \right] -1 
	\\
	0 < ( j + 2 ) - i \leqq \left[ \frac m2 \right] ,
	& \quad
	0 < ( j + 2 ) - ( i + 1 ) \leqq \left[ \frac m2 \right] - 1
	.
\end{align*}
Hence
\begin{align*}
	&
	\| \Delta_i^j \vecf \|_{ \mathbb{R}^n } \geqq
	\frac { L_m L_1 } { L_2 }
	\frac { j - i } m
	,
	\quad
	\| \Delta_i^{ j+1}  \vecf \|_{ \mathbb{R}^n } \geqq
	\frac { L_m L_1 } { L_2 }
	\frac { j + 1 - i } m
	,
	\\
	&
	\| \Delta_{ i+1 }^j \vecf \|_{ \mathbb{R}^n } \geqq
	\frac { L_m L_1 } { L_2 }
	\frac { j - i - 1 } m
	,
	\quad
	\| \Delta_{ i+1 }^{ j+1 } \vecf \|_{ \mathbb{R}^n } \geqq
	\frac { L_m L_1 } { L_2 }
	\frac { j - i } m
	\\
	&
	\| \Delta_i^{ j+2 } \vecf \|_{ \mathbb{R}^n } \geqq
	\frac { L_m L_1 } { L_2 }
	\frac { j + 2 - i } m
	,
	\quad
	\| \Delta_{ i+1 }^{ j+2 } \vecf \|_{ \mathbb{R}^n } \geqq
	\frac { L_m L_1 } { L_2 }
	\frac { j + 1 - i } m
\end{align*}
hold.
Calculating the arithmetic mean,
we obtain
\begin{align*}
	A_i \| \Delta_i^j \vecf \|_{ \mathbb{R}^n } \geqq & \
	\frac { L_m L_1 } { L_2 }
	\frac { j - i - \frac 12 } m
	,
	\\
	A_i \| \Delta_i^{ j+1 } \vecf \|_{ \mathbb{R}^n } \geqq & \
	\frac { L_m L_1 } { L_2 }
	\frac { j - i + \frac 12 } m
	,
	\\
	A_{ij} \| \Delta_i^j \vecf \|_{ \mathbb{R}^n } \geqq & \
	\frac { L_m L_1 } { L_2 }
	\frac { j - i } m
	,
	\\
	A_{ij} \| \Delta_i^{ j+1 } \vecf \|_{ \mathbb{R}^n } \geqq & \
	\frac { L_m L_1 } { L_2 }
	\frac { j + 1 - i } m
	.
\end{align*}
\par
In case of $ j = i + [ \frac m2 ] - 1 $,
it follows from
\begin{align*}
	j - i = \left[ \frac m2 \right] -1 ,
	& \quad
	j + 1 - i = \left[ \frac m2 \right] ,
	\\
	j -  i - 1 = \left[ \frac m2 \right] - 2 ,
	& \quad
	( j + 1 ) - ( i + 1) = \left[ \frac m2 \right] - 1,
	\\
	( j + 2 ) - i = \left[ \frac m2 \right] + 1 ,
	& \quad
	( j + 2 ) - ( i + 1 ) = \left[ \frac m2 \right]
\end{align*}
that
\begin{align*}
	&
	\| \Delta_i^j \vecf \|_{ \mathbb{R}^n } \geqq
	\frac { L_m L_1 } { L_2 }
	\frac { j - i } m
	,
	\quad
	\| \Delta_i^{ j+1}  \vecf \|_{ \mathbb{R}^n } \geqq
	\frac { L_m L_1 } { L_2 }
	\frac { j + 1 - i } m
	,
	\\
	&
	\| \Delta_{ i+1 }^j \vecf \|_{ \mathbb{R}^n } \geqq
	\frac { L_m L_1 } { L_2 }
	\frac { j - i - 1 } m
	,
	\quad
	\| \Delta_{ i+1 }^{ j+1 } \vecf \|_{ \mathbb{R}^n } \geqq
	\frac { L_m L_1 } { L_2 }
	\frac { j - i } m
	,
	\\
	&
	\| \Delta_i^{ j+2 } \vecf \|_{ \mathbb{R}^n } \geqq
	\frac { L_m L_1 } { L_2 }
	\frac { m + i  - j - 2 } m
	\geqq
	\frac { L_m L_1 } { L_2 }
	\frac { [ \frac m2 ] - 1 } m
	=
	\frac { L_m L_1 } { L_2 }
	\frac { j - i } m
	,
	\\
	&
	\| \Delta_{ i+1 }^{ j+2 } \vecf \|_{ \mathbb{R}^n } \geqq
	\frac { L_m L_1 } { L_2 }
	\frac { j + 1 - i } m
	.
\end{align*}
Therefore we have
\begin{align*}
	A_i \| \Delta_i^j \vecf \|_{ \mathbb{R}^n } \geqq & \
	\frac { L_m L_1 } { L_2 }
	\frac { j - i - \frac 12 } m
	,
	\\
	A_i \| \Delta_i^{ j+1 } \vecf \|_{ \mathbb{R}^n } \geqq & \
	\frac { L_m L_1 } { L_2 }
	\frac { j - i + \frac 12 } m
	,
	\\
	A_{ij} \| \Delta_i^j \vecf \|_{ \mathbb{R}^n } \geqq & \
	\frac { L_m L_1 } { L_2 }
	\frac { j - i } m
	,
	\\
	A_{ij} \| \Delta_i^j \vecf \|_{ \mathbb{R}^n } \geqq & \
	\frac { L_m L_1 } { L_2 }
	\frac { j - i + \frac 12 } m
	.
\end{align*}
\par
When $ j = i + [ \frac m2 ] $,
we have
\begin{align*}
	j - i = \left[ \frac m2 \right] ,
	& \quad
	j +1 - i = \left[ \frac m2 \right] + 1 ,
	\\
	j - i - 1 = \left[ \frac m2 \right] - 1 ,
	& \quad
	( j + 1 ) - ( i + 1) = \left[ \frac m2 \right]
	\\
	( j + 2 ) - i = \left[ \frac m2 \right] + 2 ,
	& \quad
	( j + 2 ) - ( i + 1 ) = \left[ \frac m2 \right] + 1
	,
\end{align*}
which yield
\begin{align*}
	&
	\| \Delta_i^j \vecf \|_{ \mathbb{R}^n } \geqq
	\frac { L_m L_1 } { L_2 }
	\frac { j - i } m
	,
	\\
	&
	\| \Delta_i^{ j+1}  \vecf \|_{ \mathbb{R}^n } \geqq
	\frac { L_m L_1 } { L_2 }
	\frac { m + i - j - 1 } m
	\geqq
	\frac { L_m L_1 } { L_2 }
	\frac { [ \frac m2 ] - 1 } m
	=
	\frac { L_m L_1 } { L_2 }
	\frac { j - i - 1 } m
	,
	\\
	&
	\| \Delta_{ i+1 }^j \vecf \|_{ \mathbb{R}^n } \geqq
	\frac { L_m L_1 } { L_2 }
	\frac { j - i } m
	,
	\quad
	\| \Delta_{ i+1 }^{ j+1 } \vecf \|_{ \mathbb{R}^n } \geqq
	\frac { L_m L_1 } { L_2 }
	\frac { j - i } m
	\\
	&
	\| \Delta_i^{ j+2 }  \vecf \|_{ \mathbb{R}^n } \geqq
	\frac { L_m L_1 } { L_2 }
	\frac { m + i - j - 2 } m
	\geqq
	\frac { L_m L_1 } { L_2 }
	\frac { [ \frac m2 ] - 2 } m
	=
	\frac { L_m L_1 } { L_2 }
	\frac { j - i - 2 } m
	,
	\\
	&
	\| \Delta_{ i+1 }^{ j+ 2 }  \vecf \|_{ \mathbb{R}^n } \geqq
	\frac { L_m L_1 } { L_2 }
	\frac { m + i + 1 - j - 2 } m
	\geqq
	\frac { L_m L_1 } { L_2 }
	\frac { [ \frac m2 ] - 1 } m
	=
	\frac { L_m L_1 } { L_2 }
	\frac { j - i - 1 } m
	,
\end{align*}
and hence
\begin{align*}
	A_i \| \Delta_i^j \vecf \|_{ \mathbb{R}^n } \geqq & \
	\frac { L_m L_1 } { L_2 }
	\frac { j - i } m
	,
	\\
	A_i \| \Delta_i^{ j+1 } \vecf \|_{ \mathbb{R}^n } \geqq & \
	\frac { L_m L_1 } { L_2 }
	\frac { j - i  - \frac 12 } m
	,
	\\
	A_{ij} \| \Delta_i^j \vecf \|_{ \mathbb{R}^n } \geqq & \
	\frac { L_m L_1 } { L_2 }
	\frac { j - i - \frac 14 } m
	,
	\\
	A_{ij} \| \Delta_i^{ j+1 } \vecf \|_{ \mathbb{R}^n } \geqq & \
	\frac { L_m L_1 } { L_2 }
	\frac { j - i - 1 } m
	.
\end{align*}
Since $ [ \frac m2 ] - 1 \geqq \frac 12 [ \frac m2 ] $ for $ m \geqq 4 $,
we have
\[
	A_{ij} \| \Delta_i^{ j+1 } \vecf \|_{ \mathbb{R}^n } \geqq
	\frac { L_m L_1 } { L_2 }
	\frac { j - i - 1 } m
	=
	\frac { L_m L_1 } { L_2 }
	\frac { [ \frac m2 ] - 1 } m
	\geqq
	\frac { L_m L_1 } { 2 L_2 }
	\frac { [ \frac m2 ] } m
	=
	\frac { L_m L_1 } { 2 L_2 }
	\frac { j - i } m
	.
\]
\par
The assertion can be proved for $ i > j $ in the same way.
\qed
\end{proof}
\subsection{The proof of convergence}
\begin{lem}
Assume that $ \vecf \in W^{ 2 , \infty } ( \mathbb{R} / \mathbb{Z} ) $ with {\rm \pref{bi-Lipschitz}},
and
{\rm \pref{equilateral}}.
Then it holds that
\[
	\lim_{ m \to \infty }
	\sum_{ i \ne j } \mathscr{P}_{1,ij}^m ( \vecf )
	=
	\mathcal{E}_1 ( \vecf )
	.
\]
\label{lem4}
\end{lem}
\begin{proof}
Recall that
\[
	\mathscr{P}_{1,ij} ( \vecf )
	=
	\frac 14
	\frac { g_{ij,d} } { g_{ij,n} }
	\left\{
	\left\|
	\frac { \Delta_i \vecf } { \| \Delta_i \vecf \|_{ \mathbb{R}^n } }
	-
	\frac { \Delta_j \vecf } { \| \Delta_j \vecf \|_{ \mathbb{R}^n } }
	\right\|_{ \mathbb{R}^n }^2
	+
	\left\|
	\frac { \Delta_{ i+1 } \vecf } { \| \Delta_{ i+1 } \vecf \|_{ \mathbb{R}^n } }
	-
	\frac { \Delta_{ j+1 } \vecf } { \| \Delta_{ j+1 } \vecf \|_{ \mathbb{R}^n } }
	\right\|_{ \mathbb{R}^n }^2
	\right\}
	.
\]
Let $ \chi_{ij} $ be the characteristic function of the set $ [ \theta_i , \theta_{ i+1 } ) \times [ \theta_j , \theta_{ j+1 } ) $.
Then
\begin{align*}
	\mathcal{E}_1^m ( \vecf )
	= & \
	\iint_{ ( \mathbb{R} / \mathbb{Z} )^2 }
	\sum_{ i,j }
	\frac 1 { 4 g_{ij,n} }
	\left\{
	\left\|
	\frac { \Delta_i \vecf } { \| \Delta_i \vecf \|_{ \mathbb{R}^n } }
	-
	\frac { \Delta_j \vecf } { \| \Delta_j \vecf \|_{ \mathbb{R}^n } }
	\right\|_{ \mathbb{R}^n }^2
	+
	\left\|
	\frac { \Delta_{ i+1 } \vecf } { \| \Delta_{ i+1 } \vecf \|_{ \mathbb{R}^n } }
	-
	\frac { \Delta_{ j+1 } \vecf } { \| \Delta_{ j+1 } \vecf \|_{ \mathbb{R}^n } }
	\right\|_{ \mathbb{R}^n }^2
	\right\}
	\\
	& \quad \quad \qquad \qquad
	\times
	\frac { \| \Delta_i \vecf \|_{ \mathbb{R}^n } } { | \Delta_i \theta | }
	\frac { \| \Delta_j \vecf \|_{ \mathbb{R}^n } } { | \Delta_j \theta | }
	\chi_{ij} ( \vartheta_1 , \vartheta_2 )
	\, d \vartheta_2 d \vartheta_2
	.
\end{align*}
It holds for
a.e.\ $( \vartheta_1 , \vartheta_2 ) $ that
\begin{align*}
	&
	\sum_{ i,j }
	\frac 1 { 4 g_{ij,n} }
	\left\{
	\left\|
	\frac { \Delta_i \vecf } { \| \Delta_i \vecf \|_{ \mathbb{R}^n } }
	-
	\frac { \Delta_j \vecf } { \| \Delta_j \vecf \|_{ \mathbb{R}^n } }
	\right\|_{ \mathbb{R}^n }^2
	+
	\left\|
	\frac { \Delta_{ i+1 } \vecf } { \| \Delta_{ i+1 } \vecf \|_{ \mathbb{R}^n } }
	-
	\frac { \Delta_{ j+1 } \vecf } { \| \Delta_{ j+1 } \vecf \|_{ \mathbb{R}^n } }
	\right\|_{ \mathbb{R}^n }^2
	\right\}
	\\
	& \quad \quad \qquad
	\times
	\frac { \| \Delta_i \vecf \|_{ \mathbb{R}^n } } { | \Delta_i \theta | }
	\frac { \| \Delta_j \vecf \|_{ \mathbb{R}^n } } { | \Delta_j \theta | }
	\chi_{ij} ( \vartheta_1 , \vartheta_2 )
	\\
	& \quad
	\to
	\frac 12
	\frac
	{ \| \vectau ( \vartheta_1 ) - \vectau ( \vartheta_2 ) \|_{ \mathbb{R}^n } }
	{ \| \vecf ( \vartheta_1 ) - \vecf ( \vartheta_2 ) \|_{ \mathbb{R}^n } }
	\| \dot { \vecf } ( \vartheta_1 ) \|_{ \mathbb{R}^n }
	\| \dot { \vecf } ( \vartheta_2 ) \|_{ \mathbb{R}^n }
\end{align*}
as $ m \to \infty $.
Since $ \dot { \vecf } $ is Lipschitz,
we have
\[
	K \left( \dot { \vecf } , \frac { 2L } { L_2 m } \right)
	\leqq
	\frac C m
\]
for some positive constant independent of $ m $.
It follows from Corollay \ref{cor1} that
\[
	\left\|
	\frac { \Delta_i \vecf } { \| \Delta_i \vecf \|_{ \mathbb{R}^n } }
	-
	\frac { \Delta_j \vecf } { \| \Delta_j \vecf \|_{ \mathbb{R}^n } }
	\right\|_{ \mathbb{R}^n }
	\leqq
	\frac { C | i - j | } m
	.
\]
A similar estimate holds for $ \displaystyle{
	\left\|
	\frac { \Delta_{ i+1 } \vecf } { \| \Delta_{ i+1 } \vecf \|_{ \mathbb{R}^n } }
	-
	\frac { \Delta_{ j+1 } \vecf } { \| \Delta_{ j+1 } \vecf \|_{ \mathbb{R}^n } }
	\right\|_{ \mathbb{R}^n }
} $.
Combining this with
\[
	\| \Delta_i^j \vecf \|_{ \mathbb{R} } \geqq \frac { C | i - j | } m
\]
for $ | i - j | \leqq \left[ \frac m2 \right] $ and \pref{bi-Lipschitz}
we obtain
\begin{align*}
	&
	\sum_{ i,j }
	\frac 1 { 4 g_{ij,n} }
	\left\{
	\left\|
	\frac { \Delta_i \vecf } { \| \Delta_i \vecf \|_{ \mathbb{R}^n } }
	-
	\frac { \Delta_j \vecf } { \| \Delta_j \vecf \|_{ \mathbb{R}^n } }
	\right\|_{ \mathbb{R}^n }^2
	+
	\left\|
	\frac { \Delta_{ i+1 } \vecf } { \| \Delta_{ i+1 } \vecf \|_{ \mathbb{R}^n } }
	-
	\frac { \Delta_{ j+1 } \vecf } { \| \Delta_{ j+1 } \vecf \|_{ \mathbb{R}^n } }
	\right\|_{ \mathbb{R}^n }^2
	\right\}
	\\
	& \quad \quad \qquad
	\times
	\frac { \| \Delta_i \vecf \|_{ \mathbb{R}^n } } { | \Delta_i \theta | }
	\frac { \| \Delta_j \vecf \|_{ \mathbb{R}^n } } { | \Delta_j \theta | }
	\chi_{ij} ( \vartheta_1 , \vartheta_2 )
	\\
	& \quad
	\leqq
	C
	\sum_{ i,j }
	\frac 1 { \left( \frac Lm | i - j |^2 \right)^2 }
	\left( \frac { | i - j | } m \right)^2
	L_2^2
	\chi_{ij} ( \vartheta_1 , \vartheta_2 )
	\leqq
	C
	.
\end{align*}
Consequently the assertion is derived from Lebesgue's convergence theorem.
\qed
\end{proof}
\par
We now show $ \displaystyle{ \sum_{ i \ne j } \mathscr{R}_{1,ij}^m ( \vecf ) \to 0 } $.
Under \pref{equilateral}
$ 	\mathrm{I}_{13} =
	\mathrm{II} =
	\mathrm{III} =
	\mathrm{IV} = 0 $.
Therefore
\[
	\mathscr{R}_{1,ij}^m ( \vecf )
	=
	\mathrm{I}_{12} + \mathrm{I}_2
\]
\begin{lem}
Suppose that $ \vecf \in H^{ \frac 32 } ( \mathbb{R} / \mathbb{Z} ) $,
$ \vectau \in H^{ \frac 12 } ( \mathbb{R} / \mathbb{Z} ) $,
{\rm \pref{bi-Lipschitz}},
and {\rm \pref{equilateral}}.
Furthermore we assume
\[
	\lim_{ \epsilon \to + 0 }
	\epsilon^{-1} K ( \dot { \vecf } , \epsilon )^2
	= 0
	.
\]
Then we have
\[
	\lim_{ m \to \infty } \sum \mathrm{I}_{12} = 0 .
\]
\end{lem}
\begin{proof}
$ \mathrm{I}_{12} $ can be written as
\begin{align*}
	\mathrm{I}_{12}
	= & \
	\frac 14
	\frac { g_{ij,d} } { g_{ij,n} }
	\left\{
	-
	\frac {
	\left\{ A_{ij}
	\left(
	\| \Delta_i^j \vecf \|_{ \mathbb{R}^n }
	- 
	\| \Delta_{ i+1 }^{ j+1 } \vecf \|_{ \mathbb{R}^n }
	\right)
	\right\}^2
	}
	{
	A_{ij} \| \Delta_i^j \vecf \|_{ \mathbb{R}^n }
	A_{ij} \| \Delta_{ i+1 }^{ j+1 } \vecf \|_{ \mathbb{R}^n }
	}
	\right.
	\\
	& \quad \qquad \qquad
	\left.
	\times
	\left(
	\frac { \Delta_i \vecf } { \| \Delta_i \vecf \|_{ \mathbb{R}^n } }
	\cdot
	\frac { \Delta_j \vecf } { \| \Delta_j \vecf \|_{ \mathbb{R}^n } }
	+
	\frac { \Delta_{ i+1 } \vecf } { \| \Delta_{ i+1 } \vecf \|_{ \mathbb{R}^n } }
	\cdot
	\frac { \Delta_{ j+1 } \vecf } { \| \Delta_{ j+1 } \vecf \|_{ \mathbb{R}^n } }
	\right)
	\right.
	\\
	& \quad \qquad \quad
	\left.
	+ \,
	\left(
	\frac 1 { A_{ij} \| \Delta_i^j \vecf \|_{ \mathbb{R}^n } }
	+
	\frac 1 { A_{ij} \| \Delta_{ i+1 }^{ j+1 } \vecf \|_{ \mathbb{R}^n } }
	\right)
	A_{ij}
	\left(
	\| \Delta_i^j \vecf \|_{ \mathbb{R}^n }
	- 
	\| \Delta_{ i+1 }^{ j+1 } \vecf \|_{ \mathbb{R}^n }
	\right)
	\right.
	\\
	& \quad \qquad \qquad
	\left.
	\times
	\left(
	\frac { \Delta_i \vecf } { \| \Delta_i \vecf \|_{ \mathbb{R}^n } }
	\cdot
	\frac { \Delta_j \vecf } { \| \Delta_j \vecf \|_{ \mathbb{R}^n } }
	-
	\frac { \Delta_{ i+1 } \vecf } { \| \Delta_{ i+1 } \vecf \|_{ \mathbb{R}^n } }
	\cdot
	\frac { \Delta_{ j+1 } \vecf } { \| \Delta_{ j+1 } \vecf \|_{ \mathbb{R}^n } }
	\right)
	\right\}
	.
\end{align*}
Now we use the formula
\[
	\veca \cdot \vecb - \vecc \cdot \vecd
	=
	- \frac 12
	\left\{
	( \veca - \vecc ) - ( \vecb - \vecd )
	\right\}
	\cdot
	( \veca - \vecb + \vecc - \vecd )
\]
for four unit vectors $ \veca $,
$ \vecb $,
$ \vecc $,
$ \vecd $ to get
\begin{align*}
	\mathrm{I}_{12}
	= & \
	\frac 14
	\frac { g_{ij,d} } { g_{ij,n} }
	\left[
	-
	\frac {
	\left\{ A_{ij}
	\left(
	\| \Delta_i^j \vecf \|_{ \mathbb{R}^n }
	- 
	\| \Delta_{ i+1 }^{ j+1 } \vecf \|_{ \mathbb{R}^n }
	\right)
	\right\}^2
	}
	{
	A_{ij} \| \Delta_i^j \vecf \|_{ \mathbb{R}^n }
	A_{ij} \| \Delta_{ i+1 }^{ j+1 } \vecf \|_{ \mathbb{R}^n }
	}
	\right.
	\\
	& \quad \qquad \qquad
	\left.
	\times
	\left(
	\frac { \Delta_i \vecf } { \| \Delta_i \vecf \|_{ \mathbb{R}^n } }
	\cdot
	\frac { \Delta_j \vecf } { \| \Delta_j \vecf \|_{ \mathbb{R}^n } }
	+
	\frac { \Delta_{ i+1 } \vecf } { \| \Delta_{ i+1 } \vecf \|_{ \mathbb{R}^n } }
	\cdot
	\frac { \Delta_{ j+1 } \vecf } { \| \Delta_{ j+1 } \vecf \|_{ \mathbb{R}^n } }
	\right)
	\right.
	\\
	& \quad \qquad \quad
	\left.
	- \,
	\frac 12
	\left(
	\frac 1 { A_{ij} \| \Delta_i^j \vecf \|_{ \mathbb{R}^n } }
	+
	\frac 1 { A_{ij} \| \Delta_{ i+1 }^{ j+1 } \vecf \|_{ \mathbb{R}^n } }
	\right)
	A_{ij}
	\left(
	\| \Delta_i^j \vecf \|_{ \mathbb{R}^n }
	- 
	\| \Delta_{ i+1 }^{ j+1 } \vecf \|_{ \mathbb{R}^n }
	\right)
	\right.
	\\
	& \quad \qquad \qquad
	\left.
	\times
	\left\{
	\left(
	\frac { \Delta_i \vecf } { \| \Delta_i \vecf \|_{ \mathbb{R}^n } }
	-
	\frac { \Delta_{ i+1 } \vecf } { \| \Delta_{ i+1 } \vecf \|_{ \mathbb{R}^n } }
	\right)
	-
	\left(
	\frac { \Delta_j \vecf } { \| \Delta_j \vecf \|_{ \mathbb{R}^n } }
	-
	\frac { \Delta_{ j+1 } \vecf } { \| \Delta_{ j+1 } \vecf \|_{ \mathbb{R}^n } }
	\right)
	\right\}
	\right.
	\\
	& \quad \qquad \qquad \qquad
	\left.
	\cdot
	\left(
	\frac { \Delta_i \vecf } { \| \Delta_i \vecf \|_{ \mathbb{R}^n } }
	-
	\frac { \Delta_j \vecf } { \| \Delta_j \vecf \|_{ \mathbb{R}^n } }
	+
	\frac { \Delta_{ i+1 } \vecf } { \| \Delta_{ i+1 } \vecf \|_{ \mathbb{R}^n } }
	-
	\frac { \Delta_{ j+1 } \vecf } { \| \Delta_{ j+1 } \vecf \|_{ \mathbb{R}^n } }   
	\right)
	\right]
	.
\end{align*}
As said befere we may assume $ 1 \leqq | i - j | \leqq \left[ \frac m2 \right] $.
From \pref{bi-Lipschitz} we have
\[
	g_{ij,d} = \left( \frac { L_m } m \right)^2 \leqq \left( \frac Lm \right)^2 .
\]
For sufficient large $ m $
\[
	g_{ij,n} \geqq \left( \frac { L L_1 | i - j | } { 2 L_2 m } \right)^2 .
\]
It follows from Corollary \ref{cor2} that
\[
	\left\{
	A_{ij} \left(
	\| \Delta_{ i+1 }^{ j+1 } \vecf \|_{ \mathbb{R}^n }
	-
	\| \Delta_i^j \vecf \|_{ \mathbb{R}^n }
	\right)
	\right\}^2
	\leqq
	\left\{
	\frac { C L | i - j | } m
	K \left( \dot { \vecf } , \frac { 2 L } { L_2 m } \right)
	\right\}^2
	.
\]
Lemma \ref{lem3} gives us the estimate
\[
	A_{ij} \| \Delta_i^j \vecf \|_{ \mathbb{R}^n }
	A_{ij} \| \Delta_{ i+1 }^{ j+1 } \vecf \|_{ \mathbb{R}^n }
	\geqq
	\left(
	\frac { L L_1 } { 2 L_2 }
	\frac { | i - j | - \frac 14} m
	\right)^2
\]
for large $ m $.
We find from these estimates that
\begin{align*}
	&
	\left|
	 \frac 12
	\frac { g_{ij,d} } { g_{ij,n} }
	\frac {
	\left\{ A_{ij}
	\left(
	\| \Delta_i^j \vecf \|_{ \mathbb{R}^n }
	- 
	\| \Delta_{ i+1 }^{ j+1 } \vecf \|_{ \mathbb{R}^n }
	\right)
	\right\}^2
	}
	{
	A_{ij} \| \Delta_i^j \vecf \|_{ \mathbb{R}^n }
	A_{ij} \| \Delta_{ i+1 }^{ j+1 } \vecf \|_{ \mathbb{R}^n }
	}
	\right.
	\\
	& \quad \qquad \qquad
	\left.
	\times
	\left(
	\frac { \Delta_i \vecf } { \| \Delta_i \vecf \|_{ \mathbb{R}^n } }
	\cdot
	\frac { \Delta_j \vecf } { \| \Delta_j \vecf \|_{ \mathbb{R}^n } }
	+
	\frac { \Delta_{ i+1 } \vecf } { \| \Delta_{ i+1 } \vecf \|_{ \mathbb{R}^n } }
	\cdot
	\frac { \Delta_{ j+1 } \vecf } { \| \Delta_{ j+1 } \vecf \|_{ \mathbb{R}^n } }
	\right)
	\right|
	\\
	\leqq & \
	\left[
	\frac Lm \frac { 2 L_2 m } { L L_1 | i - j | }
	\left\{
	\frac { C L | i - j |} m
	K \left( \dot { \vecf } , \frac { 2 L } { L_2 m } \right)
	\right\}
	\frac { 2 L_2 m }
	{ L L_1 ( | i - j |- \frac 14 ) }
	\right]^2
	\\
	\leqq & \
	\frac C { ( | i - j |- \frac 14 )^2 }
	K \left( \dot { \vecf } , \frac { 2 L } { L_2 m } \right)^2
	.
\end{align*}
On the other hand \ref{cor1} implies
\[
	\left\|
	\frac { \Delta_i \vecf } { \| \Delta_i \vecf \|_{ \mathbb{R}^n } }
	-
	\frac { \Delta_{ i+1 } \vecf } { \| \Delta_{ i+1 } \vecf \|_{ \mathbb{R}^n } }
	\right\|_{ \mathbb{R}^n }
	\leqq
	C K \left( \dot { \vecf } , \frac { 2 L } { L_2 m } \right)
	.
\]
Therefore
\begin{align*}
	&
	\left|
	\frac 14
	\frac { g_{ij,d} } { g_{ij,n} }
	\left(
	\frac 1 { A_{ij} \| \Delta_i^j \vecf \|_{ \mathbb{R}^n } }
	+
	\frac 1 { A_{ij} \| \Delta_{ i+1 }^{ j+1 } \vecf \|_{ \mathbb{R}^n } }
	\right)
	A_{ij}
	\left(
	\| \Delta_i^j \vecf \|_{ \mathbb{R}^n }
	- 
	\| \Delta_{ i+1 }^{ j+1 } \vecf \|_{ \mathbb{R}^n }
	\right)
	\right.
	\\
	& \quad \qquad
	\left.
	\left.
	\times
	\left\{
	\left(
	\frac { \Delta_i \vecf } { \| \Delta_i \vecf \|_{ \mathbb{R}^n } }
	-
	\frac { \Delta_{ i+1 } \vecf } { \| \Delta_{ i+1 } \vecf \|_{ \mathbb{R}^n } }
	\right)
	-
	\left(
	\frac { \Delta_j \vecf } { \| \Delta_j \vecf \|_{ \mathbb{R}^n } }
	-
	\frac { \Delta_{ j+1 } \vecf } { \| \Delta_{ j+1 } \vecf \|_{ \mathbb{R}^n } }
	\right)
	\right\}
	\right.
	\right.
	\\
	& \quad \qquad \qquad
	\left.
	\left.
	\cdot
	\left(
	\frac { \Delta_i \vecf } { \| \Delta_i \vecf \|_{ \mathbb{R}^n } }
	-
	\frac { \Delta_j \vecf } { \| \Delta_j \vecf \|_{ \mathbb{R}^n } }
	+
	\frac { \Delta_{ i+1 } \vecf } { \| \Delta_{ i+1 } \vecf \|_{ \mathbb{R}^n } }
	-
	\frac { \Delta_{ j+1 } \vecf } { \| \Delta_{ j+1 } \vecf \|_{ \mathbb{R}^n } }   
	\right)
	\right]
	\right|
	\\
	& \quad
	\leqq
	\frac Lm \frac { 2 L_2 m } { L L_1 | i - j | }
	\left\{
	\frac { C L | i - j |} m
	K \left( \dot { \vecf } , \frac { 2 L } { L_2 m } \right)
	\right\}
	K \left( \dot { \vecf } , \frac { 2 L } { L_2 m } \right)
	\left(
	\frac { 2 L_2 m }
	{ L L_1 ( | i - j |- \frac 14 ) }
	\right)^2
	\\
	& \quad
	\leqq
	\frac C { ( | i - j |- \frac 14 )^2 }
	K \left( \dot { \vecf } , \frac { 2 L } { L_2 m } \right)^2
	.
\end{align*}
As a result we obtain
\begin{align*}
	\sum | \mathrm{I}_{12} |
	\leqq & \
	\sum_{ i=1 }^m
	\sum_{ k=1 }^{ \left[ \frac m2 \right] }
	\frac C { ( k - \frac 14 )^2 }
	K \left( \dot { \vecf } , \frac { 2 L } { L_2 m } \right)^2
	\\
	\leqq & \
	C m
	K \left( \dot { \vecf } , \frac { 2 L } { L_2 m } \right)^2
	\to 
	0	\quad ( m \to \infty ) .
\end{align*}
\qed
\end{proof}
\par
To estimate $ \displaystyle{ \sum_{ i \ne j } \mathrm{I}_2 } $,
we decomposed $ \mathrm{I}_2 $ into
\begin{align*}
	\mathrm{I}_2
	= & \
	\mathrm{I}_{21}
	+
	\mathrm{I}_{22}
	,
	\\
	\mathrm{I}_{21}
	= & \
	-
	\frac 12
	\frac
	{
	A_{ij}
	\left(
	\| \Delta_{ i+1 }^{ j+1 } \vecf \|_{ \mathbb{R}^n }
	-
	\| \Delta_i^j \vecf \|_{ \mathbb{R}^n }
	\right)
	}
	{ g_{ij,n} }
	\\
	& \quad
	\times
	\left\{
	\frac {
	\left( \Delta_i A_j \| \Delta_i^j \vecf \|_{ \mathbb{R}^n } \right)
	\left( \Delta_j A_i \| \Delta_i^j \vecf \|_{ \mathbb{R}^n } \right)
	}
	{ A_{ij} \| \Delta_i^j \vecf \|_{ \mathbb{R}^n } }
	-
	\frac
	{
	\left( \Delta_i A_j \| \Delta_{ i+1 }^{ j+1 } \vecf \|_{ \mathbb{R}^n } \right)
	\left( \Delta_j A_i \| \Delta_{ i+1 }^{ j+1 } \vecf \|_{ \mathbb{R}^n } \right)
	}
	{ A_{ij} \| \Delta_{ i+i }^{ j+1 } \vecf \|_{ \mathbb{R}^n } }
	\right\}
	,
	\\
	\mathrm{I}_{22}
	= & \
	-
	\frac 12
	\frac
	{
	\left\{
	\Delta_i A_j
	\left(
	\| \Delta_{ i+1 }^{ j+1 } \vecf \|_{ \mathbb{R}^n }
	-
	\| \Delta_i^j \vecf \|_{ \mathbb{R}^n }
	\right)
	\right\}
	\left\{
	\Delta_j A_i
	\left(
	\| \Delta_{ i+1 }^{ j+1 } \vecf \|_{ \mathbb{R}^n }
	-
	\| \Delta_i^j \vecf \|_{ \mathbb{R}^n }
	\right)
	\right\}
	}
	{ g_{ij,n} }
	.
\end{align*}
Here we use \pref{equilateral}.
To estimate each part,
we need the following lemma.
\begin{lem}
Assume that $ m $ is sufficient large.
When $ | i - j | \leqq [ \frac m2 ] $,
\[
	\left|
	\Delta_j A_i
	\left(
	\| \Delta_i^j \vecf \|_{ \mathbb{R}^n }
	-
	\| \Delta_{ i+1 }^{ j+1 } \vecf \|_{ \mathbb{R}^n }
	\right)
	\right|
	\leqq
	\frac { CL } m
	K \left( \dot { \vecf } , \frac { 2 L } { L_2 m } \right)
\]
holds.
\end{lem}
\begin{proof}
We have
\begin{align*}
	&
	\Delta_j A_i
	\left(
	\| \Delta_i^j \vecf \|_{ \mathbb{R}^n }
	-
	\| \Delta_{ i+1 }^{ j+1 } \vecf \|_{ \mathbb{R}^n }
	\right)
	\\
	& \quad
	=
	A_i
	\left[
	\frac {
	\left\{
	A_j
	\left(
	\Delta_i^j \vecf	- \Delta_{ i+1 }^{ j+1 } \vecf
	\right)
	\right\}
	\cdot \Delta_j \vecf }
	{ A_j \| \Delta_i^j \vecf \|_{ \mathbb{R}^n } }
	+
	\frac {
	A_j \Delta_{ i+1 }^{ j+1 } \vecf \cdot
	\left( \Delta_j \vecf - \Delta_{ j+1 } \vecf \right)
	}
	{ A_j \| \Delta_i^j \vecf \|_{ \mathbb{R}^n } }
	\right.
	\\
	& \quad \qquad \qquad
	\left.
	+ \,
	\frac {
	\left\{
	A_j
	\left(
	\| \Delta_{ i+1 }^{ j+1 } \vecf \|_{ \mathbb{R}^n }
	-
	\| \Delta_i^j \vecf \|_{ \mathbb{R}^n }
	\right)
	\right\}
	A_j \Delta_{ i+1 }^{ j+1 } \vecf \cdot \Delta_{ j+1 } \vecf
	}
	{
	\left( A_j \| \Delta_i^j \vecf \|_{ \mathbb{R}^n } \right)
	\left( A_j \| \Delta_{ i+1 }^{ j+1 } \vecf \|_{ \mathbb{R}^n } \right)
	}
	\right]
	.
\end{align*}
We want to estimate each term.
\par
It holds that
\[
	\Delta_i A_j
	\| \Delta_i^j \vecf \|_{ \mathbb{R}^n }
	=
	A_j \left(
	\frac { A_i \Delta_i^j \vecf \cdot \Delta_i \vecf }
	{ A_i \| \Delta_i \vecf \|_{ \mathbb{R}^n } }
	\right)
	=
	A_j \left(
	\frac { A_i \Delta_i^j \vecf }
	{ A_i \| \Delta_i \vecf \|_{ \mathbb{R}^n } }
	\right)
	\cdot \Delta_i \vecf
	.
\]
Since
\[
	\left\|
	\frac { A_i \Delta_i^j \vecf }
	{ A_i \| \Delta_i \vecf \|_{ \mathbb{R}^n } }
	\right\|_{ \mathbb{R}^n }
	\leqq
	\frac {
	\| \Delta_i^j \vecf + \Delta_{ i+1 }^j \vecf \|_{ \mathbb{R}^n }
	}
	{
	\| \Delta_i^j \vecf \|_{ \mathbb{R}^n }
	+
	\| \Delta_{ i+1 }^j \vecf \|_{ \mathbb{R}^n }
	}
	\leqq
	1
	,
\]
we get
\[
	\left|
	\Delta_i A_j
	\| \Delta_i^j \vecf \|_{ \mathbb{R}^n }
	\right|
	\leqq
	\| \Delta_i \vecf \|_{ \mathbb{R}^n }
	.
\]
Corollary \ref{cor1} yields
\[
	\left|
	\| \Delta_{ i+1 }^{ j+1 } \vecf - \Delta_i^j \vecf \|_{ \mathbb{R}^n }
	\right|
	\leqq
	\frac { CL | i - j |} m
	K \left( \dot { \vecf } , \frac { 2 L } { L_2 m } \right)
	,
\]
and hence
\[
	\left\|
	A_j
	\left(
	\Delta_{ i+1 }^{ j+1 } \vecf - \Delta_i^j \vecf
	\right)
	\right\|_{ \mathbb{R}^n }
	\leqq
	\frac { CL | i - j | + \frac 12} m
	K \left( \dot { \vecf } , \frac { 2 L } { L_2 m } \right)
	\leqq
	\frac { CL | i - j | } m
	K \left( \dot { \vecf } , \frac { 2 L } { L_2 m } \right)
	.
\]
Here we use $ \frac 12 < | i - j | $,
which is obvious from $ | i - j | \geqq 1 $.
Consequently we obtain
\begin{align*}
	&
	\left\|
	A_i
	\left[
	\frac {
	\left\{
	A_j
	\left(
	\Delta_i^j \vecf	- \Delta_{ i+1 }^{ j+1 } \vecf
	\right)
	\right\}
	\cdot \Delta_j \vecf }
	{ A_j \| \Delta_i^j \vecf \|_{ \mathbb{R}^n } }
	\right]
	\right\|_{ \mathbb{R}^n }
	\\
	& \quad
	\leqq
	A_i
	\left(
	\frac { CL | i - j | } m
	K \left( \dot { \vecf } , \frac { 2 L } { L_2 m } \right)
	\frac 1 { A_j \| \Delta_i^j \vecf \|_{ \mathbb{R}^n } }
	\right)
	\| \Delta_j \vecf \|_{ \mathbb{R}^n }
	\\
	& \quad
	\leqq
	\frac { C L^2 } { m^2 }
	K \left( \dot { \vecf } , \frac { 2 L } { L_2 m } \right)
	A_i \left(
	\frac { | i - j | } { A_j \| \Delta_i^j \vecf \|_{ \mathbb{R}^n } }
	\right)
	.
\end{align*}
Lemma \ref{lem3} shows
\begin{align*}
	A_i \left(
	\frac { | i - j | } { A_j \| \Delta_i^j \vecf \|_{ \mathbb{R}^n } }
	\right)
	\leqq 
	\frac { L_2 m ( | i - j | + \frac 12 ) } { L L_1 ( | i - j | - \frac 12 ) }
	\leqq 
	\frac { C m } L
	.
\end{align*}
Therefore we have
\begin{align*}
	\left\|
	A_i
	\left[
	\frac {
	\left\{
	A_j
	\left(
	\Delta_i^j \vecf	- \Delta_{ i+1 }^{ j+1 } \vecf
	\right)
	\right\}
	\cdot \Delta_j \vecf }
	{ A_j \| \Delta_i^j \vecf \|_{ \mathbb{R}^n } }
	\right]
	\right\|_{ \mathbb{R}^n }
	\leqq
	\frac { C L } m
	K \left( \dot { \vecf } , \frac { 2 L } { L_2 m } \right)
	.
\end{align*}
\par
Since
\[
	\| \Delta_{ i+1 }^{ j+1 } \vecf \|_{ \mathbb{R}^n } \leqq \frac { L | i - j | } m
	,
	\quad
	\| \Delta_{ i+1 }^{ j+2 } \vecf \|_{ \mathbb{R}^n } \leqq \frac { L ( | i - j | + 1 ) } m
	\leqq
	\frac { 2 L | i - j | } m
\]
for $ | i - j | \leqq [ \frac m2 ] $,
we have
\[
	A_j \| \Delta_{ i+1 }^{ j+1 } \vecf \|_{ \mathbb{R}^n } \leqq \frac { C L | i - j | } m .
\]
By Lemma \ref{estimate_Delta(i+1)-Deltai} we get
\[
	\| \Delta_j \vecf - 	\Delta_{ j+1 } \vecf \|_{ \mathbb{R}^n }
	\leqq
	\frac { CL } m
	K \left( \dot { \vecf } , \frac { 2 L } { L_2 m } \right) .
\]
Consequently it holds that
\begin{align*}
	&
	\left\|
	A_i
	\left(
	\frac {
	A_j \Delta_{ i+1 }^{ j+1 } \vecf \cdot
	\left( \Delta_j \vecf - \Delta_{ j+1 } \vecf \right)
	}
	{ A_j \| \Delta_i^j \vecf \|_{ \mathbb{R}^n } }
	\right)
	\right\|_{ \mathbb{R}^n }
	\\
	& \quad
	\leqq
	A_i \left(
	\frac { C L | i - j | } m
	\frac { CL } m
	K \left( \dot { \vecf } , \frac { 2 L } { L_2 m } \right)
	\frac 1 { A_j \| \Delta_i^j \vecf \|_{ \mathbb{R}^n } }
	\right)
	\\
	& \quad
	\leqq
	\frac { C L } m
	K \left( \dot { \vecf } , \frac { 2 L } { L_2 m } \right)
	.
\end{align*}
\par
The estimate
\begin{align*}
	&
	\left\|
	A_i \left[
	\frac {
	\left\{
	A_j
	\left(
	\| \Delta_{ i+1 }^{ j+1 } \vecf \|_{ \mathbb{R}^n }
	-
	\| \Delta_i^j \vecf \|_{ \mathbb{R}^n }
	\right)
	\right\}
	A_j \Delta_{ i+1 }^{ j+1 } \vecf \cdot \Delta_{ j+1 } \vecf
	}
	{
	\left( A_j \| \Delta_i^j \vecf \|_{ \mathbb{R}^n } \right)
	\left( A_j \| \Delta_{ i+1 }^{ j+1 } \vecf \|_{ \mathbb{R}^n } \right)
	}
	\right]
	\right\|_{ \mathbb{R}^n }
	\\
	& \quad
	\leqq
	A_i
	\left(
	\left\{
	\frac { CL | i - j | } m
	K \left( \dot { \vecf } , \frac { 2 L } { L_2 m } \right)
	\right\}
	\frac
	{
	\| A_j \Delta_{ i+1 }^{ j+1 } \vecf \|_{ \mathbb{R}^n }
	\| \Delta_{ j+1 } \vecf \|_{ \mathbb{R}^n }
	}
	{
	\left( A_j \| \Delta_i^j \vecf \|_{ \mathbb{R}^n } \right)
	\left( A_j \| \Delta_{ i+i }^{ j+1 } \vecf \|_{ \mathbb{R}^n } \right)
	}
	\right)
	\\
	& \quad
	\leqq
	\frac { CL^3 } { m^3 }
	K \left( \dot { \vecf } , \frac { 2 L } { L_2 m } \right)
	A_i \left(
	\frac { | i - j |^2 }
	{
	\left( A_j \| \Delta_i^j \vecf \|_{ \mathbb{R}^n } \right)
	\left( A_j \| \Delta_{ i+i }^{ j+1 } \vecf \|_{ \mathbb{R}^n } \right)
	}
	\right)
\end{align*}
can be shown similarly.
Hence we have
\begin{align*}
	&
	A_i \left(
	\frac { | i - j |^2 }
	{
	\left( A_j \| \Delta_i^j \vecf \|_{ \mathbb{R}^n } \right)
	\left( A_j \| \Delta_{ i+1 }^{ j+1 } \vecf \|_{ \mathbb{R}^n } \right)
	}
	\right)
	\\
	& \quad
	\leqq
	\left[
	\left\{ \frac { m | i - j | } { L ( | i - j | - \frac 12 ) } \right\}^2
	+
	\left\{ \frac { m | i - j | } { L ( | i - j | - \frac 12 ) } \right\}^2
	\right]
	\leqq
	C \left( \frac mL \right)^2
	.
\end{align*}
From this we arrive at
\begin{align*}
	&
	\left\|
	A_i \left[
	\frac
	{
	\left\{
	A_{ij}
	\left(
	\| \Delta_{ i+1 }^{ j+1 } \vecf \|_{ \mathbb{R}^n }
	-
	\| \Delta_i^j \vecf \|_{ \mathbb{R}^n }
	\right)
	\right\}
	\left( \Delta_i A_j \| \Delta_{ i+1 }^{ j+1 } \vecf \|_{ \mathbb{R}^n } \right)
	\left( \Delta_j A_i \| \Delta_{ i+1 }^{ j+1 } \vecf \|_{ \mathbb{R}^n } \right)
	}
	{
	\left( A_{ij} \| \Delta_i^j \vecf \|_{ \mathbb{R}^n } \right)
	\left( A_{ij} \| \Delta_{ i+i }^{ j+1 } \vecf \|_{ \mathbb{R}^n } \right)
	}
	\right]
	\right\|_{ \mathbb{R}^n }
	\\
	& \quad
	\leqq
	\frac { CL } m
	K \left( \dot { \vecf } , \frac { 2 L } { L_2 m } \right)
	.
\end{align*}
\qed
\end{proof}
\begin{lem}
Suppose that $ \vecf \in H^{ \frac 32 } ( \mathbb{R} / \mathbb{Z} ) $,
$ \vectau \in H^{ \frac 12 } ( \mathbb{R} / \mathbb{Z} ) $,
{\rm \pref{bi-Lipschitz}},
and {\rm \pref{equilateral}}.
Furthermore we assume
\[
	\lim_{ \epsilon \to + 0 }
	\epsilon^{-1} K ( \dot { \vecf } , \epsilon )^2
	= 0
	.
\]
Then we have
\[
	\lim_{ m \to \infty } \sum \mathrm{I}_{21} =
	\lim_{ m \to \infty } \sum \mathrm{I}_{22} = 0 .
\]
\end{lem}
\begin{proof}
Terms in braces of $ \mathrm{I}_{21} $ can be rewritten as
\begin{align*}
	&
	\frac {
	\left( \Delta_i A_j \| \Delta_i^j \vecf \|_{ \mathbb{R}^n } \right)
	\left( \Delta_j A_i \| \Delta_i^j \vecf \|_{ \mathbb{R}^n } \right)
	}
	{ A_{ij} \| \Delta_i^j \vecf \|_{ \mathbb{R}^n } }
	-
	\frac
	{
	\left( \Delta_i A_j \| \Delta_{ i+1 }^{ j+1 } \vecf \|_{ \mathbb{R}^n } \right)
	\left( \Delta_j A_i \| \Delta_{ i+1 }^{ j+1 } \vecf \|_{ \mathbb{R}^n } \right)
	}
	{ A_{ij} \| \Delta_{ i+i }^{ j+1 } \vecf \|_{ \mathbb{R}^n } }
	\\
	& \quad
	=
	\frac {
	\left( \Delta_i A_j \| \Delta_i^j \vecf \|_{ \mathbb{R}^n } \right)
	\left\{
	\Delta_j A_i
	\left(
	\| \Delta_i^j \vecf \|_{ \mathbb{R}^n }
	-
	\| \Delta_{ i+1 }^{ j+1 } \vecf \|_{ \mathbb{R}^n }
	\right)
	\right\}
	}
	{ A_{ij} \| \Delta_i^j \vecf \|_{ \mathbb{R}^n } }
	\\
	& \quad \qquad
	+ \,
	\frac
	{
	\left\{
	\Delta_i A_j
	\left(
	\| \Delta_i^j \vecf \|_{ \mathbb{R}^n }
	-
	\| \Delta_{ i+1 }^{ j+1 } \vecf \|_{ \mathbb{R}^n }
	\right)
	\right\}
	\left( \Delta_j A_i \| \Delta_{ i+1 }^{ j+1 } \vecf \|_{ \mathbb{R}^n } \right)
	}
	{ A_{ij} \| \Delta_i^j \vecf \|_{ \mathbb{R}^n } }
	\\
	& \quad \qquad
	+ \,
	\frac
	{
	\left\{
	A_{ij}
	\left(
	\| \Delta_{ i+1 }^{ j+1 } \vecf \|_{ \mathbb{R}^n }
	-
	\| \Delta_i^j \vecf \|_{ \mathbb{R}^n }
	\right)
	\right\}
	\left( \Delta_i A_j \| \Delta_{ i+1 }^{ j+1 } \vecf \|_{ \mathbb{R}^n } \right)
	\left( \Delta_j A_i \| \Delta_{ i+1 }^{ j+1 } \vecf \|_{ \mathbb{R}^n } \right)
	}
	{
	\left( A_{ij} \| \Delta_i^j \vecf \|_{ \mathbb{R}^n } \right)
	\left( A_{ij} \| \Delta_{ i+i }^{ j+1 } \vecf \|_{ \mathbb{R}^n } \right)
	}
	.
\end{align*}
\par
It follows from Lammas \ref{estimate_Delta(i+1)-Deltai},
\ref{lem3} and
\[
	\left| \Delta_i A_j \| \Delta_i^j \vecf \|_{ \mathbb{R}^n } \right|
	\leqq
	\| \Delta_i \vecf \|_{ \mathbb{R}^n }
	=
	\frac Lm
\]
that
\[
	\left|
	\frac {
	\left( \Delta_i A_j \| \Delta_i^j \vecf \|_{ \mathbb{R}^n } \right)
	\left\{
	\Delta_j A_i
	\left(
	\| \Delta_i^j \vecf \|_{ \mathbb{R}^n }
	-
	\| \Delta_{ i+1 }^{ j+1 } \vecf \|_{ \mathbb{R}^n }
	\right)
	\right\}
	}
	{ A_{ij} \| \Delta_i^j \vecf \|_{ \mathbb{R}^n } }
	\right|
	\leqq
	\frac { CL } { m ( | i - j | - \frac 14 ) }
	K \left( \dot { \vecf } , \frac { 2 L } { L_2 m } \right)
	.
\]
Similarly we have
\[
	\left|
	\frac
	{
	\left\{
	\Delta_i A_j
	\left(
	\| \Delta_i^j \vecf \|_{ \mathbb{R}^n }
	-
	\| \Delta_{ i+1 }^{ j+1 } \vecf \|_{ \mathbb{R}^n }
	\right)
	\right\}
	\left( \Delta_j A_i \| \Delta_{ i+1 }^{ j+1 } \vecf \|_{ \mathbb{R}^n } \right)
	}
	{ A_{ij} \| \Delta_i^j \vecf \|_{ \mathbb{R}^n } }
	\right|
	\leqq
	\frac { CL } { m ( | i - j | - \frac 14 ) }
	K \left( \dot { \vecf } , \frac { 2 L } { L_2 m } \right)
	.
\]
Consequently we have
\begin{align*}
	&
	\left|
	\frac
	{
	\left\{
	A_{ij}
	\left(
	\| \Delta_{ i+1 }^{ j+1 } \vecf \|_{ \mathbb{R}^n }
	-
	\| \Delta_i^j \vecf \|_{ \mathbb{R}^n }
	\right)
	\right\}
	\left( \Delta_i A_j \| \Delta_{ i+1 }^{ j+1 } \vecf \|_{ \mathbb{R}^n } \right)
	\left( \Delta_j A_i \| \Delta_{ i+1 }^{ j+1 } \vecf \|_{ \mathbb{R}^n } \right)
	}
	{
	\left( A_{ij} \| \Delta_i^j \vecf \|_{ \mathbb{R}^n } \right)
	\left( A_{ij} \| \Delta_{ i+i }^{ j+1 } \vecf \|_{ \mathbb{R}^n } \right)
	}
	\right|
	\\
	& \quad
	\leqq
	\frac { C L } { m ( | i - j | - \frac 14 )^2 }
	K \left( \dot { \vecf } , \frac { 2 L } { L_2 m } \right)
	,
\end{align*}
and then
\begin{align*}
	| \mathrm{I}_{21} |
	\leqq & \
	\frac 12
	\frac { \left| A_{ij} \left(
	\| \Delta_{ i+1 }^{ j+1 } \vecf \|_{ \mathbb{R}^n }
	- 
	\| \Delta_i^j \vecf \|_{ \mathbb{R}^n }
	\right) \right| }
	{ g_{ij,n} }
	\frac { C L } { m ( | i - j | - \frac 14 ) }
	K \left( \dot { \vecf } , \frac { 2 L } { L_2 m } \right)
	\\
	\leqq & \
	\left\{
	\frac { C L } m
	K \left( \dot { \vecf } , \frac { 2 L } { L_2 m } \right)
	\right\}^2
	\left( \frac m { L | i - j | } \right)^2
	\frac 1 { ( | i - j | - \frac 14 ) }
	\\
	\leqq & \
	\frac C { ( | i - j | - \frac 14 )^3 }
	K \left( \dot { \vecf } , \frac { 2 L } { L_2 m } \right)^2
	.
\end{align*}
Thus we know
\begin{align*}
	\sum_{ i \ne j } | \mathrm{I}_{21} |
	\leqq & \
	\sum_{ i=1 }^m
	\sum_{ k=1 }^{ [ \frac m2 ] }
	\frac C { ( k - \frac 14 )^3 }
	K \left( \dot { \vecf } , \frac { 2 L } { L_2 m } \right)^2
	\\
	\leqq & \
	C m
	K \left( \dot { \vecf } , \frac { 2 L } { L_2 m } \right)^2
	\to 
	0 \quad ( m \to \infty ) .
\end{align*}
Similarly we have
\begin{align*}
	| \mathrm{I}_{22} |
	\leqq 
	\frac 12
	\frac 1 { g_{ij,n} }
	\left\{
	\frac { C L } m
	K \left( \dot { \vecf } , \frac { 2 L } { L_2 m } \right)
	\right\}^2
	\leqq 
	\frac C { ( | i - j | - \frac 14 )^2 }
	K \left( \vectau , \frac L { L_2 m } \right)^2
	,
\end{align*}
and
\[
	\sum_{ i \ne j } | \mathrm{I}_{22} |
	\to 0
	\quad ( m \to \infty ) .
\]
\qed
\end{proof}
\par
Next we show the convergence of $ \displaystyle{ \sum_{ i \ne j } \mathscr{P}_{ij}^m ( \vecf ) } $ and $ \displaystyle{ \sum_{ i \ne j } \mathscr{R}_{ij}^m ( \vecf ) } $.
Note that we have $ \mathrm{J}_2 = \mathrm{J}_3 = \mathrm{J}_4 = \mathrm{J}_5 = 0 $ under \pref{equilateral},
and hence
\[
	\mathscr{P}_{ij}^m ( \vecf )
	= \mathrm{J}_{11} ,
	\quad
	\mathscr{R}_{ij}^m ( \vecf )
	= \mathrm{J}_{12}
	+ \mathrm{J}_{13}
	+ \mathrm{J}_{14}
	+ \mathrm{J}_{15}
	.
\]
\begin{lem}
Assume that $ \vecf \in W^{ 2 , \infty } ( \mathbb{R} / \mathbb{Z} ) $ with {\rm \pref{bi-Lipschitz}},
and
{\rm \pref{equilateral}}.
Then it holds that
\[
	\lim_{ m \to \infty }
	\sum_{ i \ne j } \mathscr{P}_{ij}^m ( \vecf )
	=
	\mathcal{E} ( \vecf ) - 4
	,
	\quad
	\lim_{ m \to \infty }
	\sum_{ i \ne j } \mathscr{R}_{ij}^m ( \vecf )
	=
	0
	.
\]
\end{lem}
\begin{proof}
The proof is based on the argument similar to those of Lemma \ref{lem4}.
It is not difficult to show
\begin{align*}
	&
	\sum_{ i,j }
	\mathscr{P}_{ij}^m ( \vecf )
	\chi_{{}_{ij} } ( \vartheta_1 , \vartheta_2 )
	\\
	& \quad
	\to
	\frac 12
	\left[
	\left\|
	\frac { \vecf ( \vartheta_1 ) - \vecf ( \vartheta_2 ) } { \| \vecf ( \vartheta_1 ) - \vecf ( \vartheta_2 ) \|_{ \mathbb{R}^n } }
	\bigwedge
	( \vectau ( \vartheta_1 ) + \vectau ( \vartheta_2 ) )
	\right\|_{ \wedge^2 \mathbb{R}^n }^2
	\right.
	\\
	& \quad \qquad
	\left.
	+ \,
	\left\{
	\frac { \vecf ( \vartheta_1 ) - \vecf ( \vartheta_2 ) } { \| \vecf ( \vartheta_1 ) - \vecf ( \vartheta_2 ) \|_{ \mathbb{R}^n } }
	\cdot
	\left( \vectau ( \vartheta_1 ) - \vectau ( \vartheta_2 ) \right)
	\right\}^2
	\right]
	\frac { \| \dot { \vecf } ( \vartheta_1 ) \|_{ \mathbb{R}^n } \| \dot { \vecf }( \vartheta_2 ) \|_{ \mathbb{R}^n } }
	{ \| \vecf ( \vartheta_1 ) - \vecf ( \vartheta_2 ) \|_{ \mathbb{R}^n }^2 }
	,
	\\
	&
	\sum_{ i,j }
	\mathscr{R}_{ij}^m ( \vecf )
	\chi_{{}_{ij} } ( \vartheta_i , \vartheta_j )
	\to 0
\end{align*}
as $ m \to \infty $ for a.e.\ $ ( \vartheta_1 , \vartheta_2 ) \in ( \mathbb{R} / \mathbb{Z} )^2 $.
The first one is the energy density of $ \mathcal{E} ( \vecf ) - 4 $.
To apply Lebesgue's theorem we will show uniform boundedness.
Let $ i < j \leqq j + \frac m2 $.
Then
\begin{align*}
	&
	\frac { \Delta_i^j \vecf }
	{ \| \Delta_i^j \vecf \|_{ \mathbb{R}^n } }
	-
	\frac { j - i }
	{ 2 \| \Delta_i^j \vecf \|_{ \mathbb{R}^n } }
	\frac { L_m } m
	\left(
	\frac { \Delta_i \vecf } { \| \Delta_i \vecf \|_{ \mathbb{R}^n } }
	+
	\frac { \Delta_j \vecf } { \| \Delta_j \vecf \|_{ \mathbb{R}^n } }
	\right)
	\\
	& \quad
	=
	\frac 1
	{ \| \Delta_i^j \vecf \|_{ \mathbb{R}^n } }
	\sum_{ k=i }^{ j-1 }
	\left\{
	\Delta_k \vecf
	-
	\frac 12
	\left(
	\Delta_i \vecf + \Delta_j \vecf
	\right)
	\right\}
	\\
	& \quad
	=
	\frac 1
	{ 2 \| \Delta_i^j \vecf \|_{ \mathbb{R}^n } }
	\sum_{ k=i }^{ j-1 }
	\left\{
	\left( \Delta_k \vecf - \Delta_i \vecf \right)
	+
	\left( \Delta_k \vecf - \Delta_j \vecf \right)
	\right\}
	.
\end{align*}
Therefore we have
\begin{align*}
	&
	\left\|
	\frac { \Delta_i^j \vecf }
	{ \| \Delta_i^j \vecf \|_{ \mathbb{R}^n } }
	-
	\frac { j - i }
	{ 2 \| \Delta_i^j \vecf \|_{ \mathbb{R}^n } }
	\frac { L_m } m
	\left(
	\frac { \Delta_i \vecf } { \| \Delta_i \vecf \|_{ \mathbb{R}^n } }
	+
	\frac { \Delta_j \vecf } { \| \Delta_j \vecf \|_{ \mathbb{R}^n } }
	\right)
	\right\|_{ \mathbb{R}^n }
	\\
	& \quad
	\leqq
	\frac C { | i - j | }
	\sum_{ k=i }^{ j-1 }
	\left\{ | k - i | + | k - j | \right\}
	K \left( \dot { \vecf } , \frac { 2 L } { L_2 m } \right)
	\\
	& \quad
	\leqq
	C | i - j | K \left( \dot { \vecf } , \frac { 2 L } { L_2 m } \right)
	.
\end{align*}
A similar estimate holds for $ i - \frac m2 \leqq j < i $ also.
Therefore it holds that
\begin{align*}
	&
	\left\|
	\frac { \Delta_i^j \vecf } { \| \Delta_i^j \vecf \|_{ \mathbb{R}^n } }
	\bigwedge
	\left(
	\frac { \Delta_i \vecf } { \| \Delta_i \vecf \|_{ \mathbb{R}^n } }
	+
	\frac { \Delta_j \vecf } { \| \Delta_j \vecf \|_{ \mathbb{R}^n } }
	\right)
	\right\|_{ \wedge^2 \mathbb{R}^n }^2
	\\
	& \quad
	=
	\left\|
	\left[
	\frac { \Delta_i^j \vecf }
	{ \| \Delta_i^j \vecf \|_{ \mathbb{R}^n } }
	-
	\frac { j - i }
	{ 2 \| \Delta_i^j \vecf \|_{ \mathbb{R}^n } }
	\frac { L_m } m
	\left(
	\frac { \Delta_i \vecf } { \| \Delta_i \vecf \|_{ \mathbb{R}^n } }
	+
	\frac { \Delta_j \vecf } { \| \Delta_j \vecf \|_{ \mathbb{R}^n } }
	\right)
	\right]
	\bigwedge
	\left(
	\frac { \Delta_i \vecf } { \| \Delta_i \vecf \|_{ \mathbb{R}^n } }
	+
	\frac { \Delta_j \vecf } { \| \Delta_j \vecf \|_{ \mathbb{R}^n } }
	\right)
	\right\|_{ \wedge^2 \mathbb{R}^n }^2
	\\
	& \quad
	\leqq
	C | i - j |^2 K \left( \dot { \vecf } , \frac { 2 L } { L_2 m } \right)^2
	.
\end{align*}
On the other hand,
we have
\[
	\left\{
	\frac { \Delta_i^j \vecf } { \| \Delta_i^j \vecf \|_{ \mathbb{R}^n } }
	\cdot
	\left(
	\frac { \Delta_i \vecf } { \| \Delta_i \vecf \|_{ \mathbb{R}^n } }
	-
	\frac { \Delta_j \vecf } { \| \Delta_j \vecf \|_{ \mathbb{R}^n } }
	\right)
	\right\}^2
	\leqq
	C | i - j |^2 K \left( \dot { \vecf } , \frac { 2 L } { L_2 m } \right)^2
	.
\]
Hence the estimate
\[
	| \mathrm{J}_{11} |
	\leqq
	\frac C { | i - j |^2 } | i - j |^2 K \left( \dot { \vecf } , \frac { 2 L } { L_2 m } \right)^2
	=
	C K \left( \dot { \vecf } , \frac { 2 L } { L_2 m } \right)^2
\]
follows from \pref{J_11}.
This implies that we can apply Lebesgue's convergence theorem to
$ \displaystyle{
	\sum_{ i \ne j } \mathrm{J}_{11}
} $,
and we find that it converges to $ \mathcal{E} ( \vecf ) - 4 $ as $ m \to \infty $.
\par
Now we estimate each term of $ \mathscr{R}_{ij}^m ( \vecf ) $.
We begin with
\[
	\frac 1 { g_{ij,n} } - \frac 1 { \bar g_{ij,n} }
	=
	\frac 1 { \| \Delta_i^j \vecf \|_{ \mathbb{R}^n } }
	\left(
	\frac 1 { \| \Delta_{ i+1 }^{ j+1 } \vecf \|_{ \mathbb{R}^n } }
	-
	\frac 1 { \| \Delta_i^j \vecf \|_{ \mathbb{R}^n } }
	\right)
	=
	\frac {
	\| \Delta_i^j \vecf \|_{ \mathbb{R}^n }
	-
	\| \Delta_{ i+1 }^{ j+1 } \vecf \|_{ \mathbb{R}^n }
	}
	{
	\| \Delta_i^j \vecf \|_{ \mathbb{R}^n }^2
	\| \Delta_{ i+1 }^{ j+1 } \vecf \|_{ \mathbb{R}^n }
	}
	.
\]
It follows from the proof of Corollary \ref{cor2} that
\[
	\left|
	\frac 1 { g_{ij,n} } - \frac 1 { \bar g_{ij,n} }
	\right|
	\leqq
	\left( \frac Lm \right)^{ -2 }
	\frac { C | i - j | } { | i - j |^3 } K \left( \dot { \vecf } , \frac { 2 L } { L_2 m } \right)
	=
	\left( \frac Lm \right)^{ -2 }
	\frac C { | i - j |^2 } K \left( \dot { \vecf } , \frac { 2 L } { L_2 m } \right)
	.
\]
Hence we can estimate $ \mathrm{J}_{12} $ as
\[
	| \mathrm{J}_{12} |
	\leqq
	\frac C { | i - j |^2 } K \left( \dot { \vecf } , \frac { 2 L } { L_2 m } \right)
	.
\]
We decompose $ \mathrm{J}_{13} $ into
\begin{align*}
	\mathrm{J}_{13}
	= & \
	\mathrm{J}_{131} + \mathrm{J}_{132}
	,
	\\
	\mathrm{J}_{131}
	= & \
	- \frac 12
	\left(
	\frac 1 { g_{ij,n} } - \frac 1 { \bar g_{ij,n} }
	\right)
	\Delta_i \Delta_j g_{ij,n}
	,
	\\
	\mathrm{J}_{132}
	= & \
	-
	\frac 1 { \bar g_{ij,n} }
	\Delta_i \Delta_j \left( g_{ij,n} - \bar g_{ij,n} \right)
	.
\end{align*}
From
\begin{align*}
	\Delta_i \Delta_j g_{ij,n}
	= & \
	-
	\Delta_i \vecf \cdot \Delta_j \vecf
	-
	\Delta_{ i+1 } \vecf \cdot \Delta_{ j+1 } \vecf
	\\
	& \quad
	-
	\frac {
	A_{ij} \left(
	\| \Delta_{ i+1 }^{ j+1 } \vecf \|_{ \mathbb{R}^n }
	-
	\| \Delta_i^j \vecf \|_{ \mathbb{R}^n }
	\right)
	}
	{
	\left( A_{ij} \| \Delta_i^j \vecf \|_{ \mathbb{R}^n } \right)
	}
	\left( \Delta_i \vecf \cdot \Delta_j \vecf \right)
	\\
	& \quad
	+
	\frac {
	A_{ij} \left(
	\| \Delta_{ i+1 }^{ j+1 } \vecf \|_{ \mathbb{R}^n }
	-
	\| \Delta_i^j \vecf \|_{ \mathbb{R}^n }
	\right)
	}
	{
	\left( A_{ij} \| \Delta_{ i+1 }^{ j+1 } \vecf \|_{ \mathbb{R}^n } \right)
	}
	\left( \Delta_{ i+1 } \vecf \cdot \Delta_{ j+1 } \vecf \right)
	\\
	& \quad
	- \,
	\frac
	{ A_{ij} \left(
	\| \Delta_{ i+1 }^{ j+1 } \vecf \|_{ \mathbb{R}^n }
	-
	\| \Delta_i^j \vecf \|_{ \mathbb{R}^n }
	\right)
	}
	{ A_{ij} \| \Delta_i^j \vecf \|_{ \mathbb{R}^n } }
	\left( \Delta_i A_j \| \Delta_i^j \vecf \|_{ \mathbb{R}^n } \right)
	\left( \Delta_j A_i \| \Delta_i^j \vecf \|_{ \mathbb{R}^n } \right)
	\\
	& \quad
	+ \,
	\frac
	{ A_{ij} \left(
	\| \Delta_{ i+1 }^{ j+1 } \vecf \|_{ \mathbb{R}^n }
	-
	\| \Delta_i^j \vecf \|_{ \mathbb{R}^n }
	\right)
	}
	{ A_{ij} \| \Delta_{ i+1 }^{ j+1 } \vecf \|_{ \mathbb{R}^n } }
	\left( \Delta_i A_j \| \Delta_{ i+1 }^{ j+1 } \vecf \|_{ \mathbb{R}^n } \right)
	\left( \Delta_j A_i \| \Delta_{ i+1 }^{ j+1 } \vecf \|_{ \mathbb{R}^n } \right)
	\\
	& \quad
	- \,
	\left\{
	\Delta_i A_j 
	\left(
	\| \Delta_{ i+1 }^{ j+1 } \vecf \|_{ \mathbb{R}^n }
	-
	\| \Delta_i^j \vecf \|_{ \mathbb{R}^n }
	\right)
	\right\}
	\left\{
	\Delta_j A_i 
	\left( 
	\| \Delta_{ i+1 }^{ j+1 } \vecf \|_{ \mathbb{R}^n }
	-
	\| \Delta_i^j \vecf \|_{ \mathbb{R}^n }
	\right)
	\right\}
\end{align*}
and $ \left| \Delta_i \Delta_j \| \Delta_i^j \vecf \|_{ \mathbb{R}^n } \right| \leqq \frac Lm $ we find
\[
	| \Delta_i \Delta_j g_{ij} | \leqq C \left( \frac Lm \right)^2 .
\]
Hence we can estimate $ \mathrm{J}_{131} $ as
\[
	| \mathrm{J}_{131} |
	\leqq
	\frac C { | i - j |^2 } K \left( \dot { \vecf } , \frac { 2 L } { L_2 m } \right)
	.
\]
Since
\[
	g_{ij,n} - \bar g_{ij,n}
	=
	\| \Delta_i^j \vecf \|_{ \mathbb{R}^n }
	\left(
	\| \Delta_{ i+1 }^{ j+1 } \vecf \|_{ \mathbb{R}^n }
	-
	\| \Delta_i^j \vecf \|_{ \mathbb{R}^n }
	\right)
	,
\]
we obtain
\begin{align*}
	&
	\Delta_j \left(
	g_{ij,n} - \bar g_{ij,n}
	\right)
	\\
	& \quad
	=
	\left( \Delta_j \| \Delta_i^j \vecf \|_{ \mathbb{R}^n } \right)
	A_j \left(
	\| \Delta_{ i+1 }^{ j+1 } \vecf \|_{ \mathbb{R}^n }
	-
	\| \Delta_i^j \vecf \|_{ \mathbb{R}^n }
	\right)
	+
	\left( A_j \| \Delta_i^j \vecf \|_{ \mathbb{R}^n } \right)
	\Delta_j
	\left(
	\| \Delta_{ i+1 }^{ j+1 } \vecf \|_{ \mathbb{R}^n }
	-
	\| \Delta_i^j \vecf \|_{ \mathbb{R}^n }
	\right)
	,
\end{align*}
and
\begin{align*}
	&
	\Delta_i \Delta_j \left(
	g_{ij,n} - \bar g_{ij,n}
	\right)
	\\
	& \quad
	=
	\left( \Delta_i \Delta_j \| \Delta_i^j \vecf \|_{ \mathbb{R}^n } \right)
	A_{ij} \left(
	\| \Delta_{ i+1 }^{ j+1 } \vecf \|_{ \mathbb{R}^n }
	-
	\| \Delta_i^j \vecf \|_{ \mathbb{R}^n }
	\right)
	\\
	& \quad \qquad
	+ \,
	\left( \Delta_j A_i \| \Delta_i^j \vecf \|_{ \mathbb{R}^n } \right)
	\Delta_i A_j \left(
	\| \Delta_{ i+1 }^{ j+1 } \vecf \|_{ \mathbb{R}^n }
	-
	\| \Delta_i^j \vecf \|_{ \mathbb{R}^n }
	\right)
	\\
	& \quad \qquad
	+ \,
	\left( \Delta_i A_j \| \Delta_i^j \vecf \|_{ \mathbb{R}^n } \right)
	\Delta_j A_i
	\left(
	\| \Delta_{ i+1 }^{ j+1 } \vecf \|_{ \mathbb{R}^n }
	-
	\| \Delta_i^j \vecf \|_{ \mathbb{R}^n }
	\right)
	\\
	& \quad \qquad
	+ \,
	\left( A_{ij} \| \Delta_i^j \vecf \|_{ \mathbb{R}^n } \right)
	\Delta_i \Delta_j
	\left(
	\| \Delta_{ i+1 }^{ j+1 } \vecf \|_{ \mathbb{R}^n }
	-
	\| \Delta_i^j \vecf \|_{ \mathbb{R}^n }
	\right)
	.
\end{align*}
We have already known
\begin{align*}
	\left|
	A_{ij} \left(
	\| \Delta_{ i+1 }^{ j+1 } \vecf \|_{ \mathbb{R}^n }
	-
	\| \Delta_i^j \vecf \|_{ \mathbb{R}^n }
	\right) \right)
	\leqq & \
	C \frac Lm | i - j | K \left( \dot { \vecf } , \frac { 2 L } { L_2 m } \right) ,
	\\
	\left|
	\Delta_i \Delta_j \| \Delta_i^j \vecf \|_{ \mathbb{R}^n }
	\right|
	\leqq & \
	\frac Lm \frac C { | i - j | }
	,
	\\
	\left| \Delta_j A_i \| \Delta_i^j \vecf \|_{ \mathbb{R}^n } \right|
	\leqq & \
	C \frac Lm
	,
	\\
	\left| A_{ij} \| \Delta_i^j \vecf \|_{ \mathbb{R}^n } \right|
	\leqq & \
	C \frac Lm | i - j |
	.
\end{align*}
Furthermore we have
\begin{align*}
	&
	\Delta_i
	\left(
	\| \Delta_{ i+1 }^{ j+1 } \vecf \|_{ \mathbb{R}^n }
	-
	\| \Delta_i^j \vecf \|_{ \mathbb{R}^n }
	\right)
	\\
	& \quad
	=
	-
	\frac { A_i \Delta_{ i+1 }^{ j+1 } \vecf \cdot \Delta_{ i+1 } \vecf }
	{ A_i \| \Delta_{ i+1 }^{ j+1 } \vecf \|_{ \mathbb{R}^n } }
	+
	\frac { A_i \Delta_i^j \vecf \cdot \Delta_i \vecf }
	{ A_i \| \Delta_i^j \vecf \|_{ \mathbb{R}^n } }
	\\
	& \quad
	=
	-
	\frac {
	A_i \left(
	\Delta_{ i+1 }^{ j+1 } \vecf - \Delta_i^j \vecf
	\right)
	\cdot \Delta_{ i+1 } \vecf }
	{ A_i \| \Delta_{ i+1 }^{ j+1 } \vecf \|_{ \mathbb{R}^n } }
	-
	\frac { A_i \Delta_i^j \vecf \cdot
	\left(
	\Delta_{ i+1 } \vecf - \Delta_i \vecf
	\right)
	}
	{ A_i \| \Delta_i^j \vecf \|_{ \mathbb{R}^n } }
	\\
	& \quad \qquad
	- \,
	\left(
	\frac 1
	{ A_i \| \Delta_{ i+1 }^{ j+1 } \vecf \|_{ \mathbb{R}^n } }
	-
	\frac 1
	{ A_i \| \Delta_i^j \vecf \|_{ \mathbb{R}^n } }
	\right)
	A_i \Delta_i^j \vecf \cdot \Delta_{ i+1 } \vecf
	.
\end{align*}
From estimates
\[
	\left|
	-
	\frac {
	A_i \left(
	\Delta_{ i+1 }^{ j+1 } \vecf - \Delta_i^j \vecf
	\right)
	\cdot \Delta_{ i+1 } \vecf }
	{ A_i \| \Delta_{ i+1 }^{ j+1 } \vecf \|_{ \mathbb{R}^n } }
	\right|
	\leqq
	C \frac Lm
	\frac { | i - j | K \left( \dot { \vecf } , \frac { 2 L } { L_2 m } \right) } { | i - j | }
	=
	C \frac Lm K \left( \dot { \vecf } , \frac { 2 L } { L_2 m } \right)
	,
\]
\[
	\left|
	-
	\frac { A_i \Delta_i^j \vecf \cdot
	\left(
	\Delta_{ i+1 } \vecf - \Delta_i \vecf
	\right)
	}
	{ A_i \| \Delta_i^j \vecf \|_{ \mathbb{R}^n } }
	\right|
	\leqq
	C \| \Delta_{ i+1 } \vecf - \Delta_i \vecf \|_{ \mathbb{R}^n }
	\leqq
	C \frac Lm K \left( \dot { \vecf } , \frac { 2 L } { L_2 m } \right)
	,
\]
\begin{align*}
	&
	\left|
	- \,
	\left(
	\frac 1
	{ A_i \| \Delta_{ i+1 }^{ j+1 } \vecf \|_{ \mathbb{R}^n } }
	-
	\frac 1
	{ A_i \| \Delta_i^j \vecf \|_{ \mathbb{R}^n } }
	\right)
	A_i \Delta_i^j \vecf \cdot \Delta_{ i+1 } \vecf
	\right|
	\\
	& \quad
	\leqq
	\left|
	\frac { A_i 
	\left(
	\| \Delta_{ i+1 }^{j+1 } \vecf \|_{ \mathbb{R}^n }
	-
	\| \Delta_i^j \vecf \|_{ \mathbb{R}^n }
	\right)
	}
	{
	A_i \| \Delta_{ i+1 }^{ j+1 } \vecf \|_{ \mathbb{R}^n }
	A_i \| \Delta_i^j \vecf \|_{ \mathbb{R}^n }
	}
	\right|
	\| A_i \Delta_i^j \vecf \|_{ \mathbb{R}^n }
	\| \Delta_{ i+1 } \vecf \|_{ \mathbb{R}^n }
	\\
	& \quad
	\leqq
	C \frac Lm
	\frac { | i - j | K \left( \dot { \vecf } , \frac { 2 L } { L_2 m } \right) | i - j | } { | i - j | }
	=
	C \frac Lm K \left( \dot { \vecf } , \frac { 2 L } { L_2 m } \right)
\end{align*}
we obtain
\[
	\left|
	\Delta_i
	\left(
	\| \Delta_{ i+1 }^{ j+1 } \vecf \|_{ \mathbb{R}^n }
	-
	\| \Delta_i^j \vecf \|_{ \mathbb{R}^n }
	\right)
	\right|
	\leqq
	C \frac Lm K \left( \dot { \vecf } , \frac { 2 L } { L_2 m } \right)
	.
\]
Since
\begin{align*}
	&
	\Delta_i \Delta_j
	\left(
	\| \Delta_{ i+1 }^{ j+1 } \vecf \|_{ \mathbb{R}^n }
	-
	\| \Delta_i^j \vecf \|_{ \mathbb{R}^n }
	\right)
	\\
	& \quad
	=
	- \frac 12
	\frac {
	\left( \Delta_{ i+1 } \vecf - \Delta_i \vecf \right)
	\cdot
	\left( \Delta_{ j+1 } \vecf + \Delta_j \vecf \right)
	+
	\left( \Delta_{ i+1 } \vecf - \Delta_i \vecf \right)
	\cdot
	\left( \Delta_{ j+1 } \vecf - \Delta_j \vecf \right)
	}
	{ A_{ij} \| \Delta_{ i+1 }^{ j+1 } \vecf \|_{ \mathbb{R}^n } }
	\\
	& \quad \qquad
	+ \,
	\left(
	\frac 1
	{ A_{ij} \| \Delta_i^j \vecf \|_{ \mathbb{R}^n } }
	-
	\frac 1
	{ A_{ij} \| \Delta_{ i+1 }^{ j+1 } \vecf \|_{ \mathbb{R}^n } }
	\right)
	\Delta_i \vecf \cdot \Delta_j \vecf
	\\
	& \quad \qquad
	- \,
	\frac 12
	\frac {
	\left\{ \Delta_i A_j
	\left(
	\| \Delta_{ i+1 }^{ j+1 } \vecf \|_{ \mathbb{R}^n }
	-
	\| \Delta_i^j \vecf \|_{ \mathbb{R}^n }
	\right) \right\}
	\left\{ \Delta_j A_i
	\left(
	\| \Delta_{ i+1 }^{ j+1 } \vecf \|_{ \mathbb{R}^n }
	+
	\| \Delta_i^j \vecf \|_{ \mathbb{R}^n }
	\right) \right\}
	}
	{ A_{ij} \| \Delta_{ i+1 }^{ j+1 } \vecf \|_{ \mathbb{R}^n } }
	\\
	& \quad \qquad
	- \,
	\frac 12
	\frac {
	\left\{ \Delta_i A_j
	\left(
	\| \Delta_{ i+1 }^{ j+1 } \vecf \|_{ \mathbb{R}^n }
	+
	\| \Delta_i^j \vecf \|_{ \mathbb{R}^n }
	\right) \right\}
	\left\{ \Delta_j A_i
	\left(
	\| \Delta_{ i+1 }^{ j+1 } \vecf \|_{ \mathbb{R}^n }
	-
	\| \Delta_i^j \vecf \|_{ \mathbb{R}^n }
	\right) \right\}
	}
	{ A_{ij} \| \Delta_{ i+1 }^{ j+1 } \vecf \|_{ \mathbb{R}^n } }
	\\
	& \quad \qquad
	+ \,
	\left(
	\frac 1
	{ A_{ij} \| \Delta_i^j \vecf \|_{ \mathbb{R}^n } }
	-
	\frac 1
	{ A_{ij} \| \Delta_{ i+1 }^{ j+1 } \vecf \|_{ \mathbb{R}^n } }
	\right)
	\left( \Delta_i A_j \| \Delta_i^j \vecf \|_{ \mathbb{R}^n } \right)
	\left( \Delta_j A_i \| \Delta_i^j \vecf \|_{ \mathbb{R}^n } \right)
	,
\end{align*}
it holds that
\[
	\left|
	\Delta_i \Delta_j
	\left(
	\| \Delta_{ i+1 }^{ j+1 } \vecf \|_{ \mathbb{R}^n }
	-
	\| \Delta_i^j \vecf \|_{ \mathbb{R}^n }
	\right)
	\right|
	\leqq
	\frac Lm \frac C { | i - j | }K \left( \dot { \vecf } , \frac { 2 L } { L_2 m } \right)
	.
\]
Combining estimates above,
we get
\[
	| \mathrm{J}_{132} | \leqq
	\frac C { | i - j |^2 } K \left( \dot { \vecf } , \frac { 2 L } { L_2 m } \right) .
\]
\par
To estimate $ \mathrm{J}_{14} $,
it is decomposed into 4 parts as
\begin{align*}
	\mathrm{J}_{14}
	= & \
	\mathrm{J}_{141} + \mathrm{J}_{142}
	,
	\\
	\mathrm{J}_{141}
	= & \
	\frac 12
	\left(
	\frac 1 { g_{ij,n}^2 }
	-
	\frac 1 { \bar g_{ij,n}^2 }
	\right)
	\left( \Delta_i g_{ij,n} \right)
	\left( \Delta_j g_{ij,n} \right)
	,
	\\
	\mathrm{J}_{142}
	= & \
	\frac 1 { 2 \bar g_{ij,n}^2 }
	\left\{
	\Delta_i \left( g_{ij,n} - \bar g_{ij,n} \right)
	\right\}
	\left( \Delta_j g_{ij,n} \right)
	,
	\\
	\mathrm{J}_{143}
	= & \
	\frac 1 { 2 \bar g_{ij,n}^2 }
	\left( \Delta_i \bar g_{ij,n} \right)
	\left\{
	\Delta_j \left( g_{ij,n} - \bar g_{ij,n} \right)
	\right\}
	,
	\\
	\mathrm{J}_{144}
	= & \
	\frac 1 { 2 \bar g_{ij,n}^2 }
	\left\{
	\left( \Delta_i \bar g_{ij,n} \right)
	\left( \Delta_j \bar g_{ij,n} \right)
	+
	4
	\left( \Delta_i^j \vecf \cdot \Delta_i \vecf \right)
	\left( \Delta_i^j \vecf \cdot \Delta_j \vecf \right)
	\right\}
	-
	\frac 1 { 2 g_{ij}^2 }
	.
\end{align*}
Since
\[
	\frac 1 { g_{ij,n}^2 }
	-
	\frac 1 { \bar g_{ij,n}^2 }
	=
	\frac 1 { \| \Delta_i^j \vecf \|_{ \mathbb{R}^n }^3 }
	\left(
	\frac 1 { \| \Delta_{ i+1 }^{ j+1 } \vecf \|_{ \mathbb{R}^n } }
	-
	\frac 1 { \| \Delta_i^j \vecf \|_{ \mathbb{R}^n } }
	\right)
	,
\]
we have
\[
	\left|
	\frac 1 { g_{ij,n}^2 }
	-
	\frac 1 { \bar g_{ij,n}^2 }
	\right|
	\leqq
	\left( \frac Lm \right)^{-4}
	\frac { C | i - j | K \left( \dot { \vecf } , \frac { 2 L } { L_2 m } \right) } { | i - j |^5 }
	=
	\left( \frac Lm \right)^{-4}
	\frac { C K \left( \dot { \vecf } , \frac { 2 L } { L_2 m } \right) } { | i - j |^4 }
	.
\]
The relation
\begin{align*}
	\Delta_i g_{ij,n}
	= 
	-
	\frac { A_i \Delta_i^j \vecf \cdot \Delta_i \vecf }
	{ A_i \| \Delta_i^j \vecf \|_{ \mathbb{R}^n } }
	A_i \| \Delta_{ i+1 }^{ j+1 } \vecf \|_{ \mathbb{R}^n }
	+
	A_i \| \Delta_i^j \vecf \|_{ \mathbb{R}^n }
	\frac { A_i \Delta_{ i+1 }^{ j+1 } \vecf \cdot \Delta_{ i+1 } \vecf }
	{ A_i \| \Delta_{ i+1 }^{ j+1 } \vecf \|_{ \mathbb{R}^n } }
\end{align*}
implies
\[
	| \Delta_i g_{ij,n} |
	\leqq
	C \left( \frac Lm \right)^2 | i - j |
	.
\]
Hence we obtain the estimate
\[
	| \mathrm{J}_{141} |
	\leqq
	\frac { CK \left( \dot { \vecf } , \frac { 2 L } { L_2 m } \right) } { | i - j |^2 }
	.
\]
It follows from
\begin{align*}
	&
	\Delta_i
	\left( g_{ij,n} - \bar g_{ij,n} \right)
	\\
	& \quad
	=
	-
	\frac { A_i \Delta_i^j \vecf \cdot \Delta_i \vecf }
	{ A_i \| \Delta_i^j \vecf \|_{ \mathbb{R}^n } }
	A_i \left( \| \Delta_{ i+1}^{ j+1 } \vecf \|_{ \mathbb{R}^n }
	- \| \Delta_i^j \vecf \|_{ \mathbb{R}^n } \right)
	\\
	& \quad \qquad
	+ \,
	\frac
	{ A_i \left( \| \Delta_{ i+1 }^{ j+1 } \vecf \|_{ \mathbb{R}^n }
	-
	\| \Delta_i^j \vecf \|_{ \mathbb{R}^n } \right) }
	{ A_i \| \Delta_{ i+1 }^{ j+1 } \vecf \|_{ \mathbb{R}^n } }
	A_i \Delta_{ i+1 }^{ j+1 } \vecf \cdot \Delta_i \vecf
	- 
	A_i \left( \Delta_{ i+1 }^{ j+1 } \vecf - \Delta_i^j \vecf \right) \cdot \Delta_i \vecf
\end{align*}
that
\[
	\left|
	\Delta_i
	\left( g_{ij,n} - \bar g_{ij,n} \right)
	\right|
	\leqq
	C \left( \frac Lm \right)^2 | i - j | K \left( \dot { \vecf } , \frac { 2 L } { L_2 m } \right)
	.
\]
Consequently
\[
	| \mathrm{J}_{142} |
	\leqq
	C \left( \frac Lm \right)^{-4} | i - j |^{-4}
	\left( \frac Lm \right)^2 | i - j | K \left( \dot { \vecf } , \frac { 2 L } { L_2 m } \right)
	\left( \frac Lm \right)^2 | i - j |
	=
	\frac { CK \left( \dot { \vecf } , \frac { 2 L } { L_2 m } \right) } { | i - j |^2 }
\]
holds.
We have
\[
	| \mathrm{J}_{143} |
	\leqq
	C \left( \frac Lm \right)^{-4} | i - j |^{-4}
	\left( \frac Lm \right)^2 | i - j |
	\left( \frac Lm \right)^2 | i - j | K \left( \dot { \vecf } , \frac { 2 L } { L_2 m } \right)
	=
	\frac { CK \left( \dot { \vecf } , \frac { 2 L } { L_2 m } \right) } { | i - j |^2 }
	.
\]
from
\[
	\Delta_i \bar g_{ij,n}
	=
	- 2 A_i \Delta_i^j \vecf \cdot \Delta_i \vecf
	.
\]
\par
Now we use
\begin{align*}
	&
	\frac 14
	\left( \Delta_i \bar g_{ij,n} \right)
	\left( \Delta_j \bar g_{ij,n} \right)
	+
	\left( \Delta_i^j \vecf \cdot \Delta_i \vecf \right)
	\left( \Delta_i^j \vecf \cdot \Delta_j \vecf \right)
	\\
	& \quad
	=
	\frac 14 \left( \frac { L_m } m \right)^2
	\left\{
	\left( \frac { L_m } m \right)^2
	+
	2 \Delta_i^j \vecf \cdot \left( \Delta_j \vecf - \Delta_i \vecf \right)
	\right\}
\end{align*}
to estimate $ \mathrm{J}_{144} $:
\[
	\sum_{ i \ne j } \mathrm{J}_{144}
	=
	\frac 12 \sum_{ i \ne j }
	\left[
	\frac 1 { \bar g_{ij,n}^2 }
	\left( \frac { L_m } m \right)^2
	\left\{
	\left( \frac { L_m } m \right)^2
	+
	2 \Delta_i^j \vecf \cdot \left( \Delta_j \vecf - \Delta_i \vecf \right)
	\right\}
	-
	\frac 1 { g_{ij}^2 }
	\right]
	.
\]
Let $ i < j \leqq i + \frac m2 $.
Then we have
\begin{align*}
	&
	\Delta_i^j \vecf \cdot ( \Delta_j \vecf - \Delta_i \vecf )
	\\
	& \quad
	=
	\sum_{ k=i }^{ j-1 } \Delta_k \vecf \cdot \sum_{ \ell = i }^{ j-1 } \Delta_\ell \Delta_\ell \vecf
	\\
	& \quad
	=
	\sum_{ k=i }^{ j-1 }	
	\sum_{ \ell=i }^{ j-1 }
	\left\{
	\Delta_k \vecf
	-
	\frac 12 \left( \Delta_\ell \vecf + \Delta_{ \ell + 1 } \vecf \right)
	\right\}
	\cdot \Delta_\ell \Delta_\ell \vecf
	\\
	& \quad
	=
	\frac 12
	\sum_{ k=i }^{ j-1 }	
	\sum_{ \ell=i }^{ j-1 }
	\left\{
	\left( \Delta_k \vecf - \Delta_\ell \vecf \right)
	+
	\left( \Delta_k \vecf - \Delta_{ \ell +1 } \vecf \right)
	\right\}
	\cdot \Delta_\ell \Delta_\ell \vecf
	\\
	& \quad
	=
	\frac 12
	\sum_{ k=i }^{ j-1 }	
	\sum_{ \ell=i }^{ j-1 }
	\left\{
	\mathrm{sgn} ( k - \ell )
	\sum_{ p = \min \{ k , \ell \} }^{ \max \{ k , \ell \}-1 } \Delta_p \Delta_p \vecf
	+
	\mathrm{sgn} ( k - \ell -1 )
	\sum_{ p = \min \{ k , \ell+1 \} }^{ \max \{ k , \ell+1 \}-1 } \Delta_p \Delta_p \vecf
	\right\}
	\cdot \Delta_\ell \Delta_\ell \vecf
	.
\end{align*}
Similar relation holds for $ i - \frac m2 \leqq j < i $.
Using
\begin{align*}
	\| \Delta_q \Delta_q \vecf \|_{ \mathbb{R}^n }
	= 
	\| \Delta_{ q+2 } \vecf - \Delta_{ q+1 } \vecf - \Delta_{ q+1 } \vecf + \Delta_q \vecf \|_{ \mathbb{R}^n }
	\leqq 
	\frac { CL } m K \left( \dot { \vecf } , \frac { 2 L } { L_2 m } \right)
	,
\end{align*}
We have
\[
	\|	\Delta_i^j \vecf \cdot ( \Delta_j - \Delta_i \|_{ \mathbb{R}^n }
	\leqq
	C \left( \frac Lm \right)^2 | i - j |^3 K \left( \dot { \vecf } , \frac { 2 L } { L_2 m } \right)^2
	.
\]
Consequently we obtain the decay 
\begin{align*}
	\sum_{ i \ne j } \frac 1 { \bar g_{ij,n} } \left( \frac { L_m } m \right)^2
	\|	\Delta_i^j \vecf \cdot ( \Delta_j - \Delta_i \|_{ \mathbb{R}^n }
	\leqq & \
	\sum_{ i \ne j } \frac { C | i - j |^3 } { | i - j |^4 } K \left( \dot { \vecf } , \frac { 2 L } { L_2 m } \right)^2
	\\
	\leqq & \
	C m \log m K \left( \dot { \vecf } , \frac { 2 L } { L_2 m } \right)^2 \to 0 \quad ( m \to \infty )
	.
\end{align*}
Furthermore we have
\begin{align*}
	\left|
	\sum_{ i \ne j }
	\left\{
	\frac 1 { 2 \bar g_{ij,n}^2 } \left( \frac { L_m } m \right)^4
	-
	\frac 1 { 2 g_{ij}^2 }
	\right\}
	\right|
	= & \
	\frac 12 \sum_{ i \ne j }
	\left( \frac { L_m } m \right)^4
	\left( \frac 1 { \bar g_{ij,n}^2 } - \frac 1 { g_{ij,n}^2 } \right)
	\\
	\leqq & \
	C
	\sum_{ i \ne j } \frac { K \left( \dot { \vecf } , \frac { 2 L } { L_2 m } \right) } { | i - j |^4 }
	\leqq
	C m K \left( \dot { \vecf } , \frac { 2 L } { L_2 m } \right)
	.
\end{align*}
If $ \vecf \in W^{1,1} ( \mathbb{R} / \mathbb{Z} ) $,
then $ m K \left( \dot { \vecf } , \frac { 2 L } { L_2 m } \right) $ is bounded unifornly on $ m $.
Thus Lebesgue's convergence theorem is applicable to
\[
	\sum_{ i \ne j } \left( \mathrm{J}_{12} + \mathrm{J}_{13} + \mathrm{J}_{141}
	+ \mathrm{J}_{142} + \mathrm{J}_{143} + \mathrm{J}_{144} \right)
	,
\]
and we find that it converges to $ 0 $ as $ m \to \infty $.
\qed
\end{proof}

\end{document}